\documentclass{article}
\usepackage{amsmath}
\usepackage{amssymb}
\usepackage{amsthm,mathtools}
\usepackage{lineno}
\usepackage[mathscr]{eucal}
\usepackage{xcolor}
\usepackage{fontenc}
\usepackage{graphicx}
\usepackage{geometry}
\usepackage{lineno}

\theoremstyle{definition}
\newtheorem{definition}{Definition}[section]
\newtheorem{theorem}[definition]{Theorem}
\newtheorem*{theorem*}{Conjecture}
\newtheorem{proposition}[definition]{Proposition}
\newtheorem{lemma}[definition]{Lemma}

\theoremstyle{remark}
\newtheorem{remark}[definition]{Remark}
\newtheorem{example}[definition]{Example}
\newcounter{enumctr}

\newcommand{\R}{\mathbb{R}}

\def\iu{{\ensuremath{\mathrm{i}}}}
\def\du{{\ensuremath{\mathrm{d}}}}

\newcommand{\diag}{\mathop{\mathrm{diag}}}
\providecommand{\keywords}[1]{\textbf{\textbf{Key words: }} #1}

\begin{document}
	\title{A constructive approach for investigating the stability of incommensurate fractional differential systems}       
	
	\author{Kai Diethelm\footnote{Faculty of Applied Natural Sciences and Humanities (FANG), 
		Technical University of Applied Sciences W\"urzburg-Schweinfurt, Ignaz-Sch\"on-Str.\ 11, 
		97421 Schweinfurt, Germany, \tt \{kai.diethelm, safoura.hashemishahraki\}@thws.de}
	\and
		Safoura Hashemishahraki$^*$
	\and
		Ha Duc Thai\footnote{Institute of Mathematics, Vietnam Academy of Science and Technology, 
			18 Hoang Quoc Viet, 10307 Ha Noi, Vietnam, \tt \{hdthai, httuan\}@math.ac.vn}
	\and
		Hoang The Tuan$^\dag$
	}
	
	\date{}
	\maketitle   
	
	\begin{abstract} 
		This paper is devoted to studying the asymptotic behaviour of solutions to 
		generalized incommensurate fractional systems. To this end,
		we first consider fractional systems with rational orders and introduce a criterion 
		that is necessary and sufficient to ensure the stability of such systems. Next, from the 
		fractional order pseudospectrum definition proposed by Šanca et al.,
		we formulate the concept of a rational approximation for the fractional 
		spectrum of a incommensurate fractional systems with general,
		not necessarily rational, orders. Our first important new
		contribution is to show the equivalence between the 
		fractional spectrum of a incommensurate linear system and its rational approximation. 
		With this result in hand, we use ideas developed in our earlier work to demonstrate the 
		stability of an equilibrium point to nonlinear systems in arbitrary finite-dimensional spaces. 
		A second novel aspect of our work is the fact that the approach is constructive. 
		It is effective and widely applicable in studying the asymptotic behavior of solutions to linear 
		incommensurate fractional differential systems with constant coefficient matrices and linearized 
		stability theory for nonlinear incommensurate fractional differential systems.
		Finally, we give numerical simulations to illustrate the merit of the proposed theoretical results.
	\end{abstract}
	
	
	
{\bf Running title:} Incommensurate fractional systems

%
%
%
%
%
%

	
\section{Introduction}
{{The primary goal of this work is to establish a deeper understanding of 
the stability of incommensurate systems of fractional differential equations
with Caputo operators. To the best of our knowledge, the first paper to investigate such questions 
was \cite{Deng} where it was shown that the system is stable if the 
zeros of its fractional characteristic polynomial are in the open left half of the 
complex plane. While this result is very valuable from a theoretical point of view,
it is only of rather limited practical use because finding roots of a fractional characteristic polynomial 
of a incommensurate fractional order system is a complicated task. 
To date, only a few studies in this direction have been carried out only in some special cases. 
A possible approach is to use the modified frequency domain analysis which is based on based on 
Nyquist's theorem or Mikhailov's stability criterion, see, e.g., \cite{Aguliar1, Sabatier, Stanis, Trige}. 
For the cases where the ordering relation of the solutions of systems is preserved (e.g., positive systems), 
modified comparison principles have been developed, see, e.g., \cite{Jia, BS_22, Tuan-Thinh, Tuan-Thinh2}.}}

Our aim in this paper is to propose a comprehensive, complete approach to solving the 
aforementioned problem. Our approach follows. First, we consider fractional order systems with 
rational orders and give a necessary and sufficient condition for their stability. 
Then we study general fractional order systems with arbitrary (rational or irrational) orders. 
Inspired by ideas in \cite{Kostic}, 
we construct rational approximations of the fractional spectrum of a matrix. The existence of 
these approximations is verified. Furthermore, we demonstrate the equivalence of the 
fractional spectrum of a matrix and its rational approximation. From this we bring 
the problem under investigation to the case when the fractional orders are rational which 
we already know how to solve clearly. The peculiarity of our approach is constructiveness.  
Based on the strategies that we propose, it is not difficult to build computer programs to check 
the stability of any fractional order system. Therefore, compared to the methods previously 
proposed in the literature, which can only solve a few specific cases, it is more effective and 
widely applicable in studying the asymptotic behavior of solutions to linear incommensurate 
fractional differential systems with constant coefficient matrices and also in linearized stability 
theory for nonlinear incommensurate fractional differential systems.

The rest of the article is organized as follows. In Section \ref{s2}, we introduce the definitions 
of fractional derivatives, fractional spectra and pseudo-spectra, and some notations that 
will be used throughout the paper. In Section \ref{s3}, we prove a necessary and sufficient 
criterion for the stability of multi-order fractional systems with rational orders. Section~\ref{s4} 
deals with generalized fractional order systems (systems containing many different,
not necessarily rational, fractional derivatives). This section contains the main results 
that are our most important contributions to understanding and solving the 
problem at hand. Next, the Mittag-Leffler stability 
of an equilibrium point of {incommensurate} fractional nonlinear systems is presented in 
Section \ref{s5}. Finally, numerical simulations are given in Section \ref{s6} to illustrate 
the obtained theoretical results.

\section{Preliminaries}
\label{s2}
	
	For $\alpha \in (0,1]$ and $J = [0, T]$ or $J = [0, \infty)$,  
	the Riemann-Liouville fractional integral of a function ${x} :J \rightarrow \mathbb R$ is defined by 
	$$ 
		I^\alpha_{0^+}x(t) 
		:= \frac{1}{\Gamma(\alpha)}\int_{0}^{t}(t-s)^{\alpha -1}x(s) \, \du s,\ \quad t\in J,
	$$
	and its Caputo fractional derivative of the order $\alpha\in (0,1)$ as 
	$$ 
		^C D^\alpha_{0^+}x(t)
		 := \frac{\du}{\du t} I^{1-\alpha}_{0^+}({x}(t) - {x}(0)), \quad t \in J \setminus \{ 0 \},
	$$
	where $\Gamma(\cdot)$ is the Gamma function and $\frac{\du}{\du t}$ is the classical derivative; see. e.g., 
	\cite[Chapters 2 and 3]{Kai} or \cite{Cong23}.
	Let $n \in \mathbb N$, $\hat\alpha = (\alpha_1, \alpha_2, \ldots, \alpha_n) \in (0,1]^n$ be a multi-index and 
	$x = (x_1,\ldots, x_n)^{\rm T}$ with $x_i : J \rightarrow \mathbb R$, $i=1,\ldots, n$, 
	be a vector valued function.
	Then we denote
	$$ 
		^C D^{\hat \alpha}_{0^+} x(t)
		 := \left(^C D^{\alpha_1}_{0^+}x_1(t), \ldots, {}^C D^{\alpha_n}_{0^+}x_n(t) \right)^{\mathrm T}.
	$$
	
	For each $n \in \mathbb N$, we denote the set of complex square matrices of 
	order $n$ by $M_n(\mathbb C)$, and 
	$M_n(\mathbb R) \subset M_n(\mathbb C )$ is the set of
	real square matrices of order $n$. The unit matrix of order $n$ is denoted by $I$.
	For a given matrix $A = (a_{ij})_{n \times n} \in M_n(\mathbb C)$, 
	we use $A^{\mathrm T} = (a_{ji})_{n \times n}$ to denote its transpose matrix and $A^* =
	(\overline{a_{ji}})_{n \times n}$ is the conjugate transpose matrix.
	For any $B\in M_n(\mathbb C)$, its spectrum is defined by 
	$\sigma (B) := \{ z \in \mathbb C : \det(z I - B) = 0 \}$. 
	Furthermore, for each $z\in \mathbb{C}$, we put $z^{\hat\alpha}I := 
	\diag(z^{\alpha_1},\ldots, z^{\alpha_n})$.
	Here and in many places later on in the paper we encounter powers of complex 
	numbers with noninteger exponents in the range $(0,1)$.
	Whenever such an expression occurs, we will interpret this in the sense of the principal branch
	of the (potentially multi-valued) complex power function, i.e.\ we say
	\[
		z^\beta = |z|^\beta \exp(\iu \beta \arg(z))
	\]
	whenever $\beta \in (0,1)$ and $z \in \mathbb C$.
	
	Next, we recall some concepts of matrix norms and pseudospectra. To simplify the notation, 
	we write $N = \{ 1, 2, \ldots, n \}$.
	On $\mathbb C^n$, we select a (for the time being, arbitrary) norm $\| \cdot\|$. 
	The associated matrix norm is also designated by $\| \cdot \|$.
	For convenience, we use the convention $\| M^{-1} \|^{-1} = 0$ if and only if $\det M = 0$. 
	For each $x \in \mathbb C^n$, we set $\Re(x) = (\Re(x_1 ), \ldots, \Re(x_n))$. 
	We denote the scalar product in $\mathbb C^n$ by $\left\langle \cdot\, ,\cdot \right\rangle$ 
	and set $\mathbb C_{-} := \{z \in \mathbb C : \Re(z) < 0 \}$ and 
	$\mathbb C_{\geq 0} := \{ z \in \mathbb C : \Re(z) \geq 0 \}$.
	
	From \cite[p.\ 248]{Kostic}, we now recall the essential concepts that we shall use to a 
	large extent throughout this paper. 

	\begin{definition}
		\label{def:pseudospectra}
		Let $n \in \mathbb N$, $A \in M_n(\mathbb R)$ and 
		$\hat\alpha = (\alpha_1, \ldots, \alpha_n)\in (0,1]^n$.
		Then, the \emph{$\hat\alpha$-order spectrum} of $A$ is the set
		\[
			\sigma_{\hat\alpha} (A)
				:= \left \{ z \in \mathbb C : 
					\det \left( \diag( z^{\alpha_1}, \ldots, z^{\alpha_n}) - A \right) = 0 \right \}.
		\]
		Moreover, for $\epsilon>0$, the \emph{$\hat\alpha$-order $\epsilon$-pseudospectrum} 
		of $A$ is defined by
		\begin{equation}
			\label{eq:pseudospectra}
			\sigma_{\hat\alpha, \epsilon}(A) := \{ z \in \mathbb C  : 
				\|(z^{\hat\alpha} I - A)^{-1}\|^{- 1} \leq \epsilon \}.
		\end{equation}
	\end{definition}
		
	It is clear from the above definition that the $\hat\alpha$-order $\epsilon$-pseudospectrum 
	depends on the used norm $\|\cdot\|$. Therefore, to indicate this dependence, we will use the 
	notation $\sigma_{\hat\alpha, \epsilon}^p(A)$ instead of $\sigma_{\hat\alpha, \epsilon}(A)$ 
	in the case where the norm $\|\cdot\|$ is specifically chosen as the norm $\|\cdot\|_p$
	with some $1 \leq p \leq \infty$. The $\hat\alpha$-order spectrum, on the other hand, 
	clearly does not depend on the chosen norm and hence does not need such a notational
	clarification.
	
	\begin{proposition}
		\label{dngp}  
		For some given $\epsilon > 0$, $\hat\alpha \in (0,1]^n$ and $A \in M_n(\mathbb R)$, the 
		$\hat\alpha$-order $\epsilon$\nobreakdash-pseudospectrum of $A$ can be expressed in the following ways:
		\begin{align}
			\sigma_{\hat\alpha, \epsilon}(A) 
				& = \{ z \in \mathbb C : \exists E \in M_n(\mathbb C), \| E \| \leq \epsilon \; 
							\text{such that}\, z \in \sigma_{\hat\alpha}(A + E) \} 
							\label{eq:pseudo-ii}\\
				& = \{ z \in \mathbb C : 
					\exists v \in \mathbb C^n, \|v\| = 1\;
					\text{such that}\, \|(z^{\hat\alpha} I - A)v\| \leq \epsilon \}.
						\label{eq:pseudo-iii}
		\end{align} 
	\end{proposition}
	
	\begin{proof} 
		See \cite[Theorem 2.3, p.\ 249]{Kostic} or \cite[Theorem 2.1, p.\ 16]{Trefethen}. 
	\end{proof}
	
	\begin{theorem}[$\hat\alpha$-fractional $\epsilon$-pseudo Ger\v sgorin sets]
		\label{1}
		Let $A \in M_n(\mathbb R)$ and $\hat \alpha \in (0,1]^n$ and consider the norm $\| \cdot \|_{\infty}$. 
		For any $\epsilon > 0$, we have
		\[
			\sigma_{\hat\alpha, \epsilon}^\infty(A)
				\subset \bigcup_{i \in N} 
				 \left\{z \in \mathbb C  : |a_{ii} - z^{\alpha_i}| \leq  r_i(A)
				 	 + \epsilon  \right\}
		\]
		where $r_i(A) =  \sum_{j\in N, j \neq i} |a_{ij}|$.
	\end{theorem}
	
	\begin{proof}
		See \cite[Theorem 3.1, p.\ 251]{Kostic}.
	\end{proof}
	
	\begin{remark}
		Taking the limit $\epsilon \to 0$, it follows from Definition \ref{def:pseudospectra} 
		and Theorem \ref{1} that
		\begin{align*}
			\sigma_{\hat\alpha}(A)\subset \bigcup_{i \in N}  
				\left\{z \in \mathbb C: 
					|a_{ii} - z^{\alpha_i}| \leq  r_i(A)  \right\}.
		\end{align*}
		Thus, if $A$ is a diagonally dominant matrix with negative elements on the main diagonal, 
		then $\sigma_{\hat\alpha}(A) \subset \mathbb C_{-}$ for all $\hat \alpha \in (0,1]^n$. 
		As shown in \cite{Tuan}, this implies that the associated linear differential equation system 
		with orders $\hat \alpha$ and the constant coefficient matrix $A$ (as given in eq.\ \eqref{eq1} 
		below) is asymptotically stable.
	\end{remark}
	
	Due to the fact that all norms on $M_n(\mathbb C)$ are equivalent, for specificity and 
	convenience of presentation, from now on we will only state and prove the results for the norm $\|\cdot\|_2$. 
	
	\begin{theorem}[Euclidean $\hat\alpha$-fractional $\epsilon$-pseudo Ger\v sgorin sets]
		\label{dlgp}
		For given $A \in M_n(\mathbb R)$, $\hat\alpha \in (0,1]^n$ and $\epsilon > 0$, we have
		\[
			\sigma_{\hat\alpha, \epsilon}^2(A)
				\subset \bigcup_{i \in N}  \left\{z \in \mathbb C : |a_{ii} - z^{\alpha_i}| 
									\leq \max\{r_i(A), r_i(A^{\mathrm T})\} + \epsilon  \right\}
		\]
	\end{theorem}
	
	\begin{proof} 
		See \cite[Theorem 3.4, pp.\ 260]{Kostic}.
	\end{proof}

	\begin{remark}
		When applying Theorem \ref{dlgp}, it is helpful to remember the immediately obvious relation
		$r_i(A^{\mathrm T}) =  \sum_{j\in N, j \neq i} |a_{ji}|$. 
	\end{remark}

\section{The $\hat{\alpha}$-order spectrum: The case $\hat\alpha \in \left( (0,1] \cap \mathbb Q \right)^n$}
	\label{s3} 
	
	Let $\hat\alpha = (\alpha_1, \ldots \alpha_n) \in (0,1]^n$. Then we initially consider the system
	\begin{align}\label{eq1}
		^C D^{\hat\alpha}_{0^+} x(t) & = Ax(t), \quad t > 0,\\
		\label{eq:ic}
		x(0) &= x^0\in\R^n,
	\end{align}
	with some $A \in M_n(\R)$.
	Following \cite{Stanis}, we shall first discuss our problem for the case that all orders 
	$\alpha_i$ are rational numbers and defer
	the extension to irrational values of $\alpha_i$ until Section \ref{s4}.
	Thus, in this section we assume that $\alpha_i \in \mathbb Q$ for all $i \in N$, 
	and so we have $\alpha_i = \frac{q_i}{m_i}$
	with some $q_i, m_i \in \mathbb N$ (assumed to be in lowest terms) for all $i \in N$. 
	Let $m$ be the least common multiple of $m_1, \ldots, m_n$ 
	and $\gamma = 1/m \in (0,1]$. Then, 
	\begin{align*}
		z^{\hat\alpha} I - A 
		& = \diag(z^{\alpha_1}, \ldots, z^{\alpha_n}) - A \\
		& = \diag(z^{\frac{q_1}{m_1}}, \ldots,  z^{\frac{q_n}{m_n}}) - A \\
		& = \diag( z^{\frac{p_1}{m}}, \ldots,  z^{\frac{p_n}{m}}) - A \\
		& = \diag( (z^\gamma)^{p_1}, \ldots, (z^\gamma)^{p_n}) - A
	\end{align*}
	where $p_i = q_i \frac{m}{m_i} \in \mathbb N$ for $i = 1, 2, \ldots, n$.  Writing
	$s=z^\gamma$,
	we obtain 
	\begin{equation}
		\label{eq:det-z-det-s}
		\det (z^{\hat\alpha} I - A) = \det (s^{\hat p} I - A)
	\end{equation}
	where $\hat p = (p_1, \ldots, p_n)^{\rm T} \in \mathbb N^n$.
	
	Since $s = z^\gamma$ in eq.\ \eqref{eq:det-z-det-s}, it is clear that $\arg(s) \in (-\gamma \pi, \gamma \pi]$. 
	Therefore, to analyze the zeros of the expression on the right-hand side of eq.\ \eqref{eq:det-z-det-s},
	it is necessary to discuss the set
	\begin{equation}
		\label{eq:def-sigmatilde}
		\tilde \sigma_{\hat p}^{(\gamma)} (A) 
		= \{s \in \mathbb C : \arg (s) \in (-\gamma \pi, \gamma \pi] \text{ and } \det (s^{\hat p} I - A) = 0 \}.
	\end{equation}
	In this context, we then see that we 
	have $\sigma_{\hat \alpha} (A) \subset \mathbb C_{-}$ if and only if 
	$\tilde \sigma_{\hat p}^{(\gamma)} (A) \;{\subset}\; \Omega_\gamma$
	where 
	\begin{equation}
		\label{eq:omegagamma}
		\Omega_\gamma 
		= \left \{ z \in \mathbb C \setminus \{ 0 \} :
				 |\arg(z)| > \gamma \frac{\pi}{2}, -\gamma \pi < \arg(z) \leq \gamma \pi \right \} .
	\end{equation}
	
	In view of eq.~\eqref{eq:def-sigmatilde}, it is thus of interest to compute $\det (s^{\hat p} I - A)$. First of all, 
	for each set $\{ i_1, i_2, \ldots, i_r \}$ with $1 \leq i_1 < \ldots < i_r \leq n$ and $1 \leq r \leq n$, 
	we will determine the coefficient of each monomial $s^{p_{i_1}}\cdots s^{p_{i_r}}$
	in the expansion of $\det (s^{\hat p} I - A)$. 
	Note that in this expansion we will treat $p_1, p_2, \ldots, p_n$ as formal variables, 
	i.e., $s^{p_i}s^{p_j} \neq s^{p_j}s^{p_i}$ with
	every $i \neq j$. We then have
	\begin{align*}
		\det (s^{\hat p} I - A) &= \det \begin{pmatrix}
							s^{p_1} - a_{11} & -a_{12} & \ldots & -a_{1n} \\
							-a_{21} &s^{p_2} - a_{22}  &  \ldots& -a_{2n} \\
							\vdots& \vdots & \ddots & \vdots \\
							-a_{n1}& -a_{n2}  &\ldots  &s^{p_n} -a_{nn}  \\
					\end{pmatrix} \\
				&= s^{p_{i_1}}\Delta^A_{(i_1;i_1)} - \sum_{j=1}^{n}(-1)^{i_1+j}a_{i_1j}\Delta^A_{(i_1;j)}
	\end{align*}
	where $\Delta^A_{(i_1;j)}$, $j \in N $, is the determinant of the matrix obtained from $s^{\hat p} I - A$ 
	by removing the $i_1$-th row and the $j$-th column. 
	It is easy to see that the term $s^{p_{i_1}}\cdots s^{p_{i_r}}$ 
	only appears in $s^{p_{i_1}}\Delta^A_{(i_1;i_1)}$. Therefore, the coefficient of 
	$s^{p_{i_1}}\cdots s^{p_{i_r}}$ 
	in the expansion $s^{\hat p} I - A$ is equal to the coefficient of 
	$s^{p_{i_1}}\cdots s^{p_{i_r}}$ in the expansion of 
	$s^{p_{i_1}}\Delta^A_{(i_1;i_1)}$. Moreover, 
	\begin{align*}
		s^{p_{i_1}}\Delta^A_{(i_1;i_1)}
		&= s^{p_{i_1}}s^{p_{i_2}}\Delta^A_{(i_1,i_2;i_1,i_2)} 
			- \sum_{j \in  N ,j \neq i_1}(-1)^{i_2+j}a_{i_2j}\Delta^A_{(i_1,i_2;i_1,j)}
	\end{align*}
	with $\Delta^A_{(i_1,i_2;i_1,j)}$, $j \in N$, $j \neq i_1$, being the determinant of the matrix 
	obtained from $s^{\hat p} I - A$ by removing  
	the rows $i_1, i_2$ and the columns $i_1, j$. Due to the fact that the term 
	$s^{p_{i_1}}s^{p_{i_2}}\cdots s^{p_{i_r}}$ 
	only appears in $s^{p_{i_1}}s^{p_{i_2}}\Delta^A_{(i_1,i_2;i_1,i_2)}$, the coefficient of 
	$s^{p_{i_1}}\cdots s^{p_{i_r}}$ in the expansion of $s^{\hat p} I - A$ is equal to the coefficient of 
	$s^{p_{i_1}}\cdots s^{p_{i_r}}$  in the expansion of
	$s^{p_{i_1}}s^{p_{i_2}}\Delta^A_{(i_1,i_2;i_1,i_2)}$.
	
	Repeating the above process, we see that the coefficient of $s^{p_{i_1}} \cdots s^{p_{i_r}}$ 
	in the expansion of $s^{\hat p} I - A$ 
	is the constant term in the expansion of $\Delta^A_{(i_1,i_2,\ldots,i_r; i_1,i_2, \ldots, i_r)}$ 
	which is the determinant of the matrix obtained from the matrix $s^{\hat p} I - A$ by removing the 
	rows $i_1, i_2, \ldots, i_r$ and the columns $i_1, i_2, \ldots, i_r$. Put
	\begin{align}\label{2.1}
		b_k = \begin{cases}
			1 & \text{ if } k = p_1 + \ldots + p_n, \\
			0 & \text{ if } k\neq p_{i_1}  + \ldots + p_{i_r}, \\
				& \phantom{\text{ if }} 1\leq i_1 < \ldots < i_r \leq n, 1 \leq r \leq n, \\
			\underset{1\leq i_1  < \ldots < i_r \leq n}{\sum} (-1)^{n-r} \det A_{(i_1, \ldots, i_r)}
				& \text{ if } k= p_{i_1} + \ldots + p_{i_r}, 1 \leq r \leq n, \\
			(-1)^n\det A & \text{ if } k = 0
		\end{cases}
	\end{align}
	where $A_{(i_1, \ldots, i_r)}$ is the matrix obtained from the matrix $A$
	by removing the $r$ rows $i_1, \ldots, i_r$ and the $r$ columns $i_1, \ldots, i_r$.
	Then, $\det (s^{\hat p} I - A) = \sum_{k=0}^{p_1 + \ldots + p_n} b_k s^k$. 

	\begin{theorem} \label{dl1} 
		Let $A \in M_n(\mathbb R)$ and $\hat\alpha = 
		(\alpha_1, \ldots, \alpha_n) \in \left( (0,1] \cap \mathbb Q \right)^n$.
		For each $i \in N$, let $\alpha_i = q_i/m_i$ with some $q_i, m_i \in \mathbb N$ 
		(in lowest terms). Let $m$ the least common multiple of $m_1, m_2, \ldots, m_n$, 
		$\hat p: = (p_1, p_2, \ldots, p_n)$ with $p_i = q_i \frac{m}{m_i}$, $\gamma: = \frac{1}{m}$ and
		$$ 
			B := \begin{pmatrix}
				0 & 0 & \cdots & 0 & -b_0 \\
				1  & 0 & \cdots &0  &-b_1  \\
				0 &1  &   &0  &-b_2  \\
				\vdots   &    & \ddots  & \vdots  & \vdots  \\
				0 & 0 & \cdots & 1 & -b_{p_1+p_2+\ldots +p_n -1} \\
			\end{pmatrix} 
		$$
		where $b_k$ is defined as in \eqref{2.1}. 
		Then, $\sigma_{\hat \alpha}(A) \subset \mathbb C_{-}$ if and only if 
		$\tilde \sigma^\gamma (B) \subset \Omega_\gamma$,
		where $\tilde \sigma^\gamma (B) := \{ s \in \mathbb C, 
		-\gamma \pi < \arg(s) \leq \gamma \pi : \det (s I - B) = 0 \}$ 
		and $\Omega_\gamma$ is as in eq.~\eqref{eq:omegagamma}.
	\end{theorem}
	
	\begin{remark}
		The question that we are interested in is to figure out whether or not a given incommensurate
		fractional order differential equation system is asymptotically stable. Recall that the classical
		criteria to establish whether or not this is true \cite{Deng} require us to find out the zeros of
		the fractional characteristic function $\det(z^{\hat \alpha} I - A)$ which is a computationally
		difficult task for which no general algorithms seem to be readily available. Our new
		Theorem~\ref{dl1} reduces this problem to finding the eigenvalues (in the classical sense) 
		of the matrix $B$. We have described an explicit method for computing this matrix,
		and it is clear that $B$ is sparse and has a very clear structure in the positioning of its
		nonzero entries. Therefore, the effective calculation of its eigenvalues may be done with 
		standard algorithms from linear algebra, thus leading to a straightforward solution of the 
		problem at hand.
	\end{remark}
	
	\begin{proof} 
		Put $P(s) = \sum_{k=0}^{p_1 + \ldots + p_n} b_k s^k$ and $s = z^{1/m}$. Then,  
		$$ 
			\det (z^{\hat\alpha} I - A) = \det (s^{\hat p} I - A) 
			= \sum_{k=0}^{p_1 + \ldots + p_n} b_k s^k = P(s) = \det (s I - B). 
		$$
		This implies that
		$$
			\det (z^{\hat\alpha} I - A) = 0 \Leftrightarrow \det (s I - B) = 0. 
		$$ 
		Thus, $\sigma_{\hat \alpha}(A) \subset \mathbb C_{-}$ 
		if and only if $\tilde \sigma^\gamma (B) \subset \Omega_\gamma$.
	\end{proof}

	\begin{remark}
		Notice that the region $|\arg(s)| > \gamma \pi$ is not
		physical which implies (keeping in mind the convention $s = z^{1/m} = z^\gamma$)
		that any root in this area of the $s$-plane 
		does not have a corresponding root in the area $-\pi<\arg(z)\leq \pi$ of the $z$-plane,
		see \cite[Subsection 2.1]{Radwan}. So, from Theorem \ref{dl1} above, 
		we actually have $\sigma_{\hat \alpha}(A) \subset \mathbb C_{-}$ if and only if 
		$\sigma (B) \subset {\tilde \Omega}_\gamma$ where 
		\[
			{\tilde \Omega}_\gamma
			:= 
			\{z \in \mathbb C\setminus\{0\} : |\arg(z)| > \gamma \frac{\pi}{2}, -\pi < \arg(z) \leq \pi \}.
		\]
	\end{remark}
	
	\begin{remark}
		When studying the asymptotic behaviour of mixed fractional order linear systems 
		where the fractional orders are rational, one can use a different approach than that 
		presented here, see \cite[Subsection 3.2]{Tuan2017}. 
		In particular, by using the semi-group property (see \cite[Chapter 8]{Kai} 
		and \cite[Subsection 4.1]{Cong23}), one can transform the original system 
		into a new equivalent system in which all fractional orders are identical to each other. 
		However, the disadvantage 
		of that approach is that the size of the derived system is often very large. In addition,
		an obvious relationship between the coefficient matrix of the original system and the 
		coefficient matrix of the derived system does not seem to be readily available.
	\end{remark}
	
	\begin{remark}
		We note that a statement similar to Theorem 3.1 was shown in the 
		survey paper by Petráš \cite[Theorem 4]{Petras09}. Our contribution here is to explicitly 
		calculate the coefficients of the characteristic polynomial $\det (s I - B)$ mentioned above 
		and clarify the proof of that result.
	\end{remark}
	
	\begin{example}\label{vd1}
		Consider the system \eqref{eq1} with 
		\begin{equation}
			\label{eq:ex1}
			A = \begin{pmatrix}
				-0.5&-0.2  &-0.15  &0.25  \\
				0.15&-0.4  &0.2  &-0.15  \\
				0.25& 0.15  &-0.6  &0.3  \\
				0.2& -0.1 &-0.1  &-0.3  \\
			\end{pmatrix} 
		\end{equation}
		and $\hat\alpha = (\frac{1}{2}, \frac{1}{4}, \frac{1}{3}, \frac{1}{6})$. 
		Then, we obtain $\gamma = \frac{1}{12}$ and $\hat p = (p_1, p_2, p_3, p_4) = (6, 3, 4, 2)$.
		
		By a direct computation, we have
		\begin{align*}
				2 &= p_4 &\hspace{3 cm} 9 &= p_1 + p_2 = p_2 + p_3 + p_4  \\
				3 &= p_2   &\hspace{3 cm} 10 &= p_1 + p_3 \\
				4 &=p_3   &\hspace{3 cm}  11 &= p_1 + p_2 + p_4 \\
				5 &= p_2 + p_4   &\hspace{3 cm}   12 &= p_1 + p_3 + p_4 \\
				6 &= p_1 = p_3 + p_4  &\hspace{3 cm}  13 &= p_1 + p_2 + p_3\\
				7 &=p_2 + p_3  &\hspace{3 cm} 15 &= p_1 + p_2 + p_3 + p_4\\
				{8} &{\; =p_1+p_4}
		\end{align*} 
		and thus $b_1 = b_{14} = 0$ and 
		\begin{align*}
			b_0 &= \det A =\frac{3759}{80000}, &
			b_2 &= -\det A_{(4)} =\frac{1211}{8000}, \\
			b_3 &= -\det A_{(2)} = \frac{203}{2000}, &
			b_4 &= -\det A_{(3)} = \frac{157}{4000},\\
			b_5 &=  \det A_{(2,4)} = \frac{27}{80}, &
			b_6 &= -\det A_{(1)} + \det A_{(3,4)} = \frac{1199}{4000},\\
			b_7 &= \det A_{(2,3)} = \frac{1}{10}, &
			b_8 &= \det A_{(1,4)} = \frac{21}{100},\\
			b_9 &= \det A_{(1,2)} - \det A_{(2,3,4)} = \frac{71}{100}, &
			b_{10} &= \det A_{(1,3)} = \frac{21}{200},\\
			b_{11} &= -\det A_{(1,2,4)} = \frac{3}{5}, &
			b_{12} &= - \det A_{(1,3,4)} = \frac{2}{5}, \\
			b_{13} &=  \det A_{(1,2,3)} = \frac{3}{10},&
			b_{15} &=1.
		\end{align*}
		Hence,
		$$ 
			B = \begin{pmatrix}
				0 & 0 & \ldots & 0 & -\frac{3759}{80000} \\
				1  & 0 & \ldots &0  &0 \\
				0 &1  &\ldots  &0  &-\frac{1211}{8000} \\
				\vdots   & \vdots  & \ddots  & \vdots  & \vdots  \\
				0 & 0 & \ldots & 1 & 0 \\
			\end{pmatrix} .
		$$
		The eigenvalues of $B$ and their arguments are 
		\begin{align*}
			\lambda_1 & \approx  -0.7521, & 	|\arg {(\lambda_1)}| & = \pi, \\
			\lambda_{2,3} & \approx -0.7822 \pm  0.4462 \iu, & |\arg {(\lambda_2)}|  = |\arg {(\lambda_3)}| & \approx 2.62319, \\ 
			\lambda_{4,5} & \approx -0.6400 \pm  0.6365 \iu, & |\arg {(\lambda_4)}|  = |\arg {(\lambda_5)}| & \approx 2.35894, \\ 
			\lambda_{6,7} & \approx -0.0087 \pm  0.9241 \iu, & |\arg {(\lambda_6)}|  = |\arg {(\lambda_7)}| & \approx 1.58021, \\ 
			\lambda_{8,9} & \approx \phantom{-}  0.7830 \pm  0.4217\iu, & |\arg {(\lambda_8)}|  = |\arg {(\lambda_9)}| & \approx 0.49402, \\ 
			\lambda_{10,11} & \approx \phantom{-} 0.6395 \pm 0.6446\iu, & |\arg {(\lambda_{10})}|  = |\arg {(\lambda_{11})}| & \approx 0.78937, \\  
			\lambda_{12,13} & \approx \phantom{-} 0.3861 \pm  0.6567\iu, & |\arg {(\lambda_{12})}|  = |\arg {(\lambda_{13})}| & \approx 1.03929, \\ 
			\lambda_{14,15} & \approx  -0.0017 \pm  0.5409 \iu, &   |\arg {(\lambda_{14})}| = |\arg {(\lambda_{15})}| & \approx 1.57393.
		\end{align*}
		This implies that $|\arg {(\lambda_i})| > \pi / 24$ for all $i = 1, \ldots, 15$.
		By Theorem~\ref{dl1}, we conclude that $\sigma_{\hat \alpha}(A) \subset \mathbb C_{-}$. 
		Thus, in this case, the system \eqref{eq1} is asymptotically stable by \cite[Theorem~1]{Deng}.
		Figure \ref{refhinh1} illustrates  this property by showing the solution to the system for 
		a certain choice of the initial value vector. In particular for $x_2$ and $x_3$, 
		one needs to compute the solutions over a very long time interval before one can 
		actually notice that the components tend to zero.
	\end{example}

	\begin{figure}[htb]
			\includegraphics[height=0.42\textwidth]{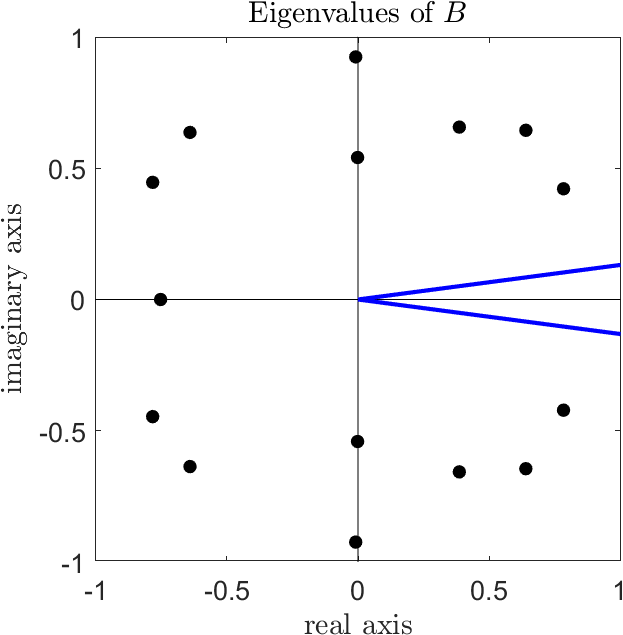}
			\hfill
			\includegraphics[height=0.42\textwidth]{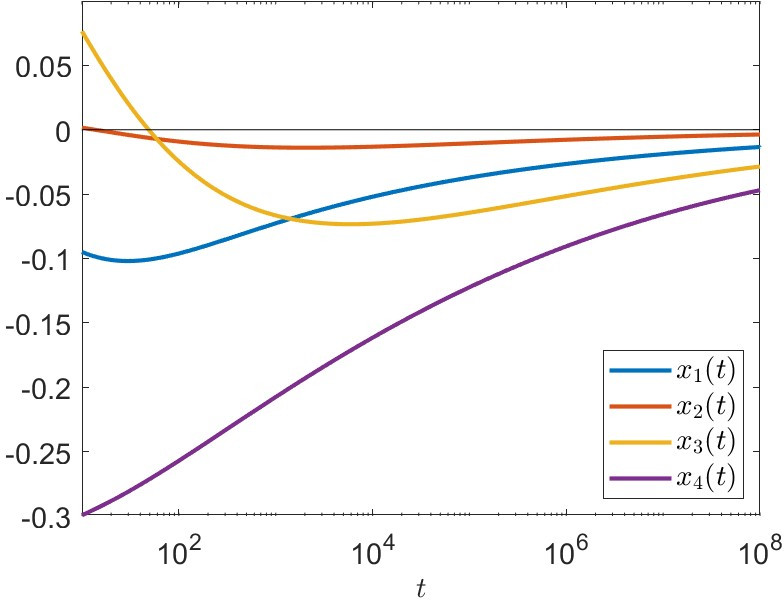}
		\caption{\emph{Left:} Location of the eigenvalues of the matrix $B$ from Example \ref{vd1} in the complex plane.
			The blue rays are oriented at an angle of $\pm \gamma \pi / 2 = \pm \pi/24$ from the 
			positive real axis and hence indicate the boundary of the critical sector 
			$\{ z \in \mathbb C : |\arg z | \le \gamma \pi /2\}$.
			Since all eigenvalues are outside of this sector, we can derive the asymptotic stability of the system.
		\emph{Right:} Trajectories of the solution of the system \eqref{eq1} 
		discussed in Example \ref{vd1} where 
		$\hat\alpha=(\frac{1}{2}, \frac{1}{4}, \frac{1}{3}, \frac{1}{6})$ 
		and the matrix $A$ is given in eq.~\eqref{eq:ex1} when the initial condition \eqref{eq:ic}
		is chosen as $x_0= (0.1, -0.1, 0.5, -0.4)^{\mathrm T}$. Note that the horizontal axis
		is displayed in a logarithmic scale.}
		\label{refhinh1}
	\end{figure}
		
	\begin{remark}
		\label{rmk:numsol}
		The solution of Example \ref{vd1} shown in the right part of Figure \ref{refhinh1} 
		has been computed numerically with Garrappa's implementation
		of the implicit product integration rule of trapezoidal type \cite{Ga2018}.
		It has been shown in \cite[Section 5]{Ga2015} that the stability properties of 
		this method are sufficient to numerically reproduce the stability of the exact solution.
		The step size {here was chosen as $h = 1$}.
		We have also used this algorithm {(but not always the same step size)} 
		for all other examples in the remainder of this paper.
	\end{remark}

\section{The $\hat\alpha$-order spectrum: The case $\hat{\alpha} \in (0,1]^n$}
\label{s4}
	
	Now we generalize our considerations to the case of systems of fractional differential equations 
	with arbitrary (not necessarily rational) orders. To this end, we first devise a strategy for replacing the 
	original (potentially irrational) orders by nearby rational numbers (see Subsection \ref{subs:rat-approx}). 
	The resulting problem
	can then be handled with the approach described in Section \ref{s3} above. Finally,
	in Subsection \ref{subs:rat-equiv} we show how to transfer the results obtained in this way 
	back to the originally given system.
	
	\subsection{Rational approximations of a fractional spectrum}
	\label{subs:rat-approx}
	
	\begin{definition}\label{dnxx}
		For a given matrix $A \in M_n(\mathbb R)$, a multi-index $\hat\alpha  
		= (\alpha_1, \ldots, \alpha_n) \in (0,1]^n$
		and  $\epsilon > 0$,
		we call $\hat\beta = (\beta_1, \ldots, \beta_n) \in {((0,1] \cap \mathbb Q)^n}$ an 
		\emph{$\epsilon$-rational approximation} of $\hat\alpha$ associated with $A$ if the 
		following conditions are satisfied:
		\begin{itemize}
			\item [(i)] $0 < \beta_i \leq \alpha_i \leq 1$ for all $i \in N$.
			\item [(ii)] There exists a constant $ R = R(A, \hat\alpha, \epsilon) \geq 1$ such that 
				\[
					\sigma_{\hat\alpha, \epsilon}^2(A) \cap \left\{z \in \mathbb C : |z| > R \right\} = 
					\sigma_{\hat\beta, \epsilon}^2(A) \cap \left\{z \in \mathbb C : |z| > R \right\} = \emptyset .
				\]
			\item [(iii)] There is a constant $\rho = \rho(A, \hat\alpha, \epsilon) \in (0, 1)$ such that
				\[
					\sigma_{\hat\alpha}(A) \cap \left\{z \in \mathbb C : |z| < \rho \right\} 
					= \sigma_{\hat\beta}(A) \cap \left\{z \in \mathbb C : |z| < \rho \right\} = \emptyset .
				\]
			\item [(iv)] For $R$ and $\rho$ chosen as above, we have
				$$ 
					\sup_{\rho \leq |z| \leq R}\, |z^{\alpha_i}- z^{\beta_i}| < \epsilon \text{ for all }i \in N. 
				$$
		\end{itemize}
	\end{definition}
	
	Our first observation in this context establishes that this definition is actually meaningful.
	
	\begin{proposition} \label{ttxx} 
		Let $A \in M_n(\mathbb R)$ such that $\det A \neq 0$ and let 
		$ \hat\alpha  = (\alpha_1, \ldots, \alpha_n) \in (0,1]^n$. 
		Then, for any  $\epsilon > 0$, there exists some $\hat\beta 
		= (\beta_1, \ldots, \beta_n) \in {((0,1] \cap \mathbb Q)^n}$ 
		which is an $\epsilon$-rational approximation of $\hat\alpha$ associated with $A$.
	\end{proposition}
	
	\begin{proof}
		Put $l(\hat\alpha) = \alpha_1 + \ldots + \alpha_n$, $\nu(\hat\alpha) = \min_{i \in N} \{ \alpha_i \}$ 
		and $\epsilon_0 = \frac{1}{2}\nu(\hat\alpha)$.
		
		We define $ \mathcal F(\hat\alpha) := \{\hat\gamma = (\hat\gamma_1,\, \ldots,
		\hat\gamma_n) \in (0,1]^n : 0 <  \alpha_i - \epsilon_0 \leq \hat\gamma_i 
		\leq \alpha_i  \text{ for all } i \in N \} $  and
		\begin{align}\label{R1}
			R:=\max \left\{ \left(\max_{i\in N} \{|a_{ii}|+ r_i(A)+r_i(A^{\mathrm T})\}
				+\epsilon\right)^{1/{(\nu(\hat\alpha)-\epsilon_0)}} ,
			 	\max_{i\in N} \left ( \frac{\epsilon}{\sqrt{2}} \right )^{1/\alpha_i}, 1\right\}
		\end{align}
		where, as in Theorem \ref{1}, we set $r_i(A) =  \sum_{j\in N, j \neq i} |a_{ij}|$.
		Then, for any $\hat\gamma \in  \mathcal F(\hat\alpha) $ and $z \in \mathbb C$, $|z| > R$,  
		we have for all $i \in N$
		\begin{align*}
			|z^{\hat\gamma_i}-a_{ii}| &\geq |z|^{\hat\gamma_i}-|a_{ii}|
					> R^{\hat\gamma_i} - |a_{ii}|\geq R^{\alpha_i -\epsilon_0}
						-|a_{ii}| \geq R^{\nu(\hat\alpha)-\epsilon_0} - |a_{ii}|\\
			& \geq |a_{ii}| +r_i(A) + r_i(A^{\mathrm T}) +\epsilon - |a_{ii}| \geq \max\{r_i(A), r_i(A^{\mathrm T})\} +\epsilon.
		\end{align*}
		Thus, by Theorem \ref{dlgp},
		\begin{align}\label{dk1}
			\sigma_{\hat\gamma, \epsilon}^2(A) \cap \left\{z \in \mathbb C : |z| > R \right\} 
			= \emptyset \quad \text{ for all } \hat\gamma \in \mathcal F(\hat\alpha).
		\end{align}
		Now take $\hat {\mathcal B_n} := \{0, 1\}^n$ 
		and $ \mathcal B_n := \left\{\xi \in \hat {\mathcal B_n} : \xi \neq (0, \ldots, 0)\;
		\text {and } \xi \neq (1, \ldots, 1) \right\}$. 
		For any $\hat{\gamma} \in \mathcal F(\hat\alpha)$, we have
		\begin{equation}\label{eq_added}
			\det (z^{\hat\gamma}I - A) 
			= z^{l(\hat\gamma)} 
				+ \sum_{\xi \in \mathcal B_n} c_\xi z^{\left< \hat\gamma, \xi\right>} +(-1)^n \det A,
		\end{equation}
		where $\left\langle \cdot\, ,\cdot \right\rangle$ is the usual scalar product on $\mathbb R^n$.
		If $\xi \in \mathcal B_n$, $\xi_{i_1} = \ldots = \xi_{i_r} = 1$ for some 
		$\{i_1, \ldots, i_r \}\subset N$ and
		$\xi_i = 0$ for all $i \in N\setminus \{i_1, \ldots, i_r\}$,
		then $z^{\left< \hat\gamma, \xi\right>} = z^{\hat\gamma_{i_1}}\cdots z^{\hat\gamma_{i_r}}$. 
		By using the same arguments as in calculating the coefficient of the term
		$s^{p_{i_1}}\cdots s^{p_{i_r}}$ in the 
		expansion of $\det (s^{\hat p} I - A)$ above, we obtain $c_\xi = (-1)^r \det A_{(i_1, \ldots, i_r)}$,
		where $A_{(i_1, \ldots, i_r)}$ is obtained from $A$ by removing the rows 
		$i_1, \ldots, i_r$ and the columns $i_1, \ldots, i_r$.
		
		Let $c = \max \{1, \max_{\xi \in \mathcal B_n} |c_\xi| \} $ and 
		$\rho_1 = \min \left\{ \left ( \frac{|\det A|}{(2^n-1)c} \right )^{1/{(\nu(\hat\alpha)-\epsilon_0)}},
		\frac{1}{2}\right\}$.
		Because $\det A \neq 0$, we may conclude that $0 <\rho_1 < 1$. 
		Moreover, for all $\hat\gamma \in \mathcal F(\hat\alpha) $ and $|z| < \rho_1 < 1$, 
		\begin{align*}
			\max \left \{ |z|^{\hat\gamma_1 + \ldots + \hat\gamma_n}, 
			\max_{\xi \in \mathcal B_n}|z|^{\left< \hat\gamma, \xi\right>} \right \} 
			& \leq |z|^{\min_{i\in N}\{\hat\gamma_i\}} \leq |z|^{\nu(\hat\alpha)-\epsilon_0}
			 < \rho_1^{\nu(\hat\alpha)-\epsilon_0} \leq \frac{|\det A|}{(2^n-1)c}.
		\end{align*}
		Hence, by \eqref{eq_added}, the following estimates hold 
		\begin{align*}
			|\det (z^{\hat\gamma} I - A)|
		              &  = \left |z^{{{\hat\gamma}_1 +\ldots {\hat\gamma}_n}} 
		              	+ \sum_{\xi \in \mathcal B_n} c_\xi z^{\left< \hat\gamma, \xi\right>} 
		              		+(-1)^n\det A \right | \\
		              & \geq |(-1)^n\det A| - |z^{{{\hat\gamma}_1 +\ldots +{\hat\gamma}_n}}| 
		              	- \left | \sum_{\xi \in \mathcal B_n} c_\xi z^{\left< \hat\gamma, \xi\right>} \right | \\
			 &	\geq |\det A| - |z|^{\hat\gamma_1 + \ldots + \hat\gamma_n} 
				- \sum_{\xi \in \mathcal B_n}|c_\xi|\cdot|z|^{\left<\hat\gamma, \xi \right>} \\
			& > |\det A| - (2^n-1)c\frac{|\det A|}{(2^n-1)c} = 0
		\end{align*}
		for any $\hat\gamma \in \mathcal F(\alpha) $ and $|z| < \rho_1 < 1$. From this, we see
		\begin{align}\label{dk2a}
			\sigma_{\hat\gamma}(A) \cap \{z \in \mathbb C: |z| < \rho_1\} = \emptyset
			\text{ for all } \hat\gamma \in \mathcal F(\hat\alpha).
		\end{align}
		Put $\rho_2 = \min \left\{\left ( \frac{\epsilon}{2} \right )^{1/(\nu(\hat\alpha)-\epsilon_0)},\frac{1}{2} \right\}$. 
		Then $0 < \rho_2 < 1$. Furthermore, for all 
		$\hat\gamma \in \mathcal F(\hat\alpha) $ and $|z| < \rho_2 < 1$, we have
		\begin{align}\label{dk2b}
			|z^{\alpha_i} - z^{\hat\gamma_i}|& \leq |z|^{\alpha_i}
			+|z|^{\hat\gamma_i} \leq  \rho_2^{\alpha_i}+\rho_2^{\hat\gamma_i} 
			\leq 2\rho_2^{\nu(\hat\alpha)-\epsilon_0}<\epsilon \text{ for all } i \in N. 
		\end{align}
		Take $\rho = \min \{\rho_1, \rho_2\}$, then $0 < \rho < 1$. 
		Moreover, from \eqref{dk2a} and \eqref{dk2b}, we conclude
		\begin{align}\label{dk2.1}
			\sup_{|z| < \rho} |z^{\alpha_i} - z^{\hat\gamma_i}| 
			& < \epsilon \text{ for all } i \in N \text{ and } \\
			\label{dk2.2}
			\sigma_{\hat\gamma}(A) \cap \{z \in \mathbb C: |z| < \rho\} 
			& = \emptyset  
			\text{ for all } \hat\gamma \in \mathcal F(\hat\alpha).
		\end{align}
		For all $z \in \mathbb C$, $\rho \leq |z| \leq R$, we use the polar coordinate form 
		$z = r(\cos\varphi + \iu \sin \varphi)$ with $\rho \leq r \leq R$ and $-\pi < \varphi \leq \pi$. Then,
		for any ${\tilde{\alpha}}={(\tilde{\alpha}_1,\dots, \tilde{\alpha}_n)} \in (0, 1]^n$, we see
		\begin{align}\label{bt}
			\left|z^{\alpha_i} -z^{{\tilde{\alpha}}_i} \right|^2 
			&= \left|\left ( r^{\alpha_i}\cos(\alpha_i\varphi)- r^{{\tilde{\alpha}_i}}\cos (\tilde\alpha_i\varphi) \right )+
				\iu \left ( r^{\alpha_i}\sin(\alpha_i\varphi)- r^{\tilde\alpha_i}\sin (\tilde\alpha_i\varphi)  \right ) \right|^2 \nonumber \\
			&= \left ( r^{\alpha_i}\cos(\alpha_i\varphi)- r^{\tilde\alpha_i}\cos (\tilde\alpha_i\varphi) \right )^2 + 
				\left ( r^{\alpha_i}\sin(\alpha_i\varphi)- r^{\tilde\alpha_i}\sin (\tilde\alpha_i\varphi)  \right )^2\nonumber \\
			&= r^{2\alpha_i} + r^{2\tilde\alpha_i} - 2r^{\alpha_i + \tilde\alpha_i}
				\left ( \cos(\alpha_i\varphi)\cos(\tilde\alpha_i\varphi) +\sin(\alpha_i\varphi)\sin(\tilde\alpha_i\varphi) \right ) \nonumber \\
			&= r^{2\alpha_i} + r^{2\tilde\alpha_i} - 2r^{\alpha_i + \tilde\alpha_i}\cos\left ( (\alpha_i -\tilde\alpha_i)\varphi \right )\nonumber \\
			&= \left ( r^{\alpha_i}-r^{\tilde\alpha_i} \right )^2 
				+ 2r^{\alpha_i + \tilde\alpha_i}\left ( 1 -\cos\left ( (\alpha_i -\tilde\alpha_i)\varphi \right )  \right ).
		\end{align}
		For each $i \in N$, we set $\delta_{1,i} = \log_R\left ( 1+ \epsilon / (\sqrt{2}R^{\alpha_i}) \right ) > 0$. 
		Then, for all $\hat{\hat \alpha} \in (0,1]^n$ such that $0 \leq \alpha_i - \hat{\hat {\alpha}}_i < \delta_{1,i}$
		for all $i \in N$ and any $\rho \leq r \leq R$, we have
		\begin{align}\label{bt1}
			r^{\alpha_i - \hat{\hat {\alpha}}_i} - 1& \leq R^{\alpha_i -\hat{\hat {\alpha}}_i}-1 
			< R^{\log_R\left ( 1+\frac{\epsilon}{\sqrt{2}R^{\alpha_i}} \right )} -1
			= \frac{\epsilon}{\sqrt{2}R^{\alpha_i}}
		\end{align}
		for all $i \in N$. 
		Because of \eqref{R1}, we know that $\epsilon / (\sqrt{2}R^{\alpha_i}) < 1$ for all $i \in N$. For each $i \in N$, let
		$\delta_{2,i} =  \log_\rho\left ( 1-\frac{\epsilon}{\sqrt{2}R^{\alpha_i}} \right ) > 0$. Then, for any
		$\tilde{\tilde{\alpha}} \in (0,1]^n$ satisfying $0 \leq \alpha_i - \tilde{\tilde{\alpha}}_i< \delta_{2,i}$ for
		all $i \in N$ and all $\rho \leq r \leq R$, we have
		\begin{align}\label{bt2}
			r^{\alpha_i - \tilde{\tilde{\alpha}}_i} - 1 \geq \rho^{\alpha_i - \tilde{\tilde{\alpha}}_i} - 1 
			> \rho^{\log_\rho \left ( 1- \epsilon / ( \sqrt{2}R^{\alpha_i}) \right )} - 1 
			= - \frac{\epsilon}{\sqrt{2}R^{\alpha_i}}.
		\end{align}
		Let ${\delta_{1,2,\min}} = \min \left\{ \delta_{1,i}, \delta_{2,i} : i \in N \right\} > 0$. 
		By combining \eqref{bt1} and \eqref{bt2}, 
		for any $\hat{\kappa} \in (0,1]^n$ such that $0 \leq \alpha_i - \hat{\kappa}_i< {\delta_{1,2,\min}}$ 
		for all $i \in N$ and $\rho \leq r \leq R$, 
		we find
		\begin{align}\label{bt3}
			- \frac{\epsilon}{\sqrt{2}R^{\alpha_i}} <  r^{\alpha_i - \hat{\kappa}_i} - 1 <  \frac{\epsilon}{\sqrt{2}R^{\alpha_i}} 
		\end{align}
		for all $i \in N$.
		Thus,  for any $\hat{\kappa} \in (0,1]^n$ with $0 \leq \alpha_i - \hat{\kappa}_i< {\delta_{1,2,\min}}$ for all $i \in N$ 
		and $\rho \leq r \leq R$, we obtain
		\begin{align}\label{bt4.0}
			\left ( r^{\alpha_i}-r^{\hat\kappa_i} \right )^2 =r^{2\hat{\kappa}_i} \left (  r^{\alpha_i - \hat{\kappa}_i} - 1 \right ) ^2
			\leq R^{2{\alpha_i}}\frac{\epsilon^2}{2R^{2\alpha_i}} = \frac{\epsilon^2}{2}
		\end{align}
		for all $i \in N$. 
		By \eqref{R1}, we have $\epsilon / (2R^{\alpha_i}) < 1$ for all $i \in N$. 
		Thus, $0 <  1 - \epsilon^2 / (4R^{2\alpha_i}) < 1$ for all $i \in N$, 
		and for each $i \in N$, there exists some $\varphi_i \in (0, \pi/2)$ such that
		$\cos \varphi_i = 1 - \epsilon^2 / (4R^{2\alpha_i})$. 
		Define $\delta_{3,i} = \varphi_i / \pi > 0$ for $i \in N$. For $\alpha^* \in (0,1]^n$ such that
		$0 \leq \alpha_i - \alpha^*_i < \delta_{3,i}$ for all $i \in N$ and $-\pi < \varphi \leq \pi$, we have
		\begin{align*}
			-\varphi_i < -\pi(\alpha_i-\alpha^*_i) \leq \varphi(\alpha_i-\alpha^*_i)
			\leq \pi(\alpha_i-\alpha^*_i)< \varphi_i
		\end{align*} 
		for all $i \in N$. Thus,
		\begin{align*}
			0 \leq 1- \cos\left ( (\alpha_i -\alpha^*_i)\varphi \right ) 
			< 1 - \cos\varphi_i = \frac{\epsilon^2}{4R^{2\alpha_i}}
		\end{align*}
		for all $i \in N$. {{This implies that for all}} 
		$\rho \leq r \leq R$ and $-\pi < \varphi \leq \pi$, we find
		\begin{align}\label{bt5}
			2r^{\alpha_i + \alpha^*_i}\left ( 1- \cos\left ( (\alpha_i -\alpha^*_i)\varphi \right ) \right ) 
				< 2R^{2\alpha_i} \frac{\epsilon^2}{4R^{2\alpha_i}} = \frac{\epsilon^2}{2},\; \forall i \in N.
		\end{align}
		Choosing ${\delta_{3,\min}} = \min_{i \in N} \left\{ \delta_{3,i} \right\}$ and 
		$\delta = \min \left\{ { \delta_{1,2,\min}, \delta_{3,\min}, } \epsilon_0 \right\}$, 
		we see that $\delta > 0$. On the other hand, using \eqref{bt}, 
		\eqref{bt4.0} and \eqref{bt5}, for any $\hat\gamma \in \mathcal F(\hat\alpha)$ such that 
		$0 \leq \alpha_i - \hat\gamma_i < \delta$ for all $i \in N$ and all $z \in \mathbb C$ with
		$\rho \leq |z| \leq R $, we have
		\begin{align}\label{dk3}
			\left|z^{\alpha_i} -z^{\hat\gamma_i} \right| < \epsilon
		\end{align}
		for all $i \in N$.
		
		Due to the density of $\mathbb Q$ in $\mathbb R$, there exists $\hat\beta \in {((0,1] \cap \mathbb Q)^n}$ such that 
		$0 \leq \alpha_i - \beta_i < \delta$ for all $i \in N$.
		We will prove that $\hat\beta$ is a rational approximation of $\hat\alpha$. Indeed, since $0 < \beta_i \leq \alpha_i \leq 1$, 
		the condition (i) in Definition \ref{dnxx} is satisfied. Since $\delta \leq \epsilon_0$, we have that
		$0 < \alpha_i -\epsilon_0 \leq \beta_i \leq \alpha_i$ for all $i\in N$. This implies
		$\hat\beta \in \mathcal F(\hat\alpha)$. Obviously $\hat\alpha \in \mathcal F(\hat \alpha)$. So, according to \eqref{dk1}, 
		\begin{align}
			\sigma_{\hat\alpha, \epsilon}^2(A) \cap \left\{z \in \mathbb C : |z| > R \right\} 
			= \emptyset 
			\text{ and }
			\sigma_{\hat\beta, \epsilon}^2(A) \cap \left\{z \in \mathbb C : |z| > R \right\} 
			= \emptyset.
		\end{align}
		Therefore, the condition (ii) in Definition \ref{dnxx} is satisfied. 
		Next, since $\hat\beta, \hat\alpha \in \mathcal F(\hat\alpha)$, by \eqref{dk2.1}, we have
		\begin{align}
			\sup_{|z| < \rho} |z^{\alpha_i} - z^{\beta_i}| < \epsilon
		\end{align}
		for all $i \in N$, and \eqref{dk2.2} implies
		\begin{align}
			\sigma_{\hat\alpha}(A) \cap \{z \in \mathbb C: |z| < \rho\} 
			= \emptyset
			\text{ and }
			\sigma_{\hat\beta}(A) \cap \{z \in \mathbb C: |z| < \rho\} = \emptyset.  
		\end{align}
		{{From this}} the condition (iii) in Definition \ref{dnxx} is satisfied.
		Finally, since $0 \leq \alpha_i - \beta_i < \delta$ for all $i \in N$, by \eqref{dk3}, we have
		\begin{align}
			\sup_{\rho \leq |z| \leq R} \,\,\left|z^{\alpha_i} -z^{\beta_i} \right| < \epsilon
		\end{align}
		for all $i \in N$.
		Hence, the condition (iv) in Definition \ref{dnxx} is satisfied.
	\end{proof}

	The above proposition actually shows us a way to find rational approximations of $\hat\alpha$ associated with a matrix $A$. 
	Indeed, based on these considerations, we can propose the following algorithm to find an $\epsilon$-rational 
	approximation of $\hat\alpha$ associated with a matrix $A$.
	
	{\bf Algorithm 1} 
	
	\hrule\vskip-0.5em
	{\bf Input}: Matrix $A$, multi-index $\hat\alpha = (\alpha_1, \ldots, \alpha_n)$ and a constant $\epsilon>0$.
	
	{\bf Step 1}: Put $a = \frac{1}{2} \min_{i = 1, \ldots ,n} \{ \alpha_i\}$ and $b = \max_{i = 1, \ldots ,n} \{ \alpha_i\}$.
	
	{\bf Step 2}: Calculate all the principal minors and the determinant of $A$. 
		Then compare the calculated numbers with each other and with $1$ to find the largest number which is then assigned to $c$. 
		
	{\bf Step 3}: Calculate the following parameters:  
	\begin{align*}
		R &=\max \left\{ \left(\max_{i\in N} \{|a_{ii}|+ r_i(A)+r_i(A^{\mathrm T})\}+\epsilon\right)^{1/a} , 
							\max_{i\in N} \left ( \frac{\epsilon}{\sqrt{2}} \right )^{1/\alpha_i}, 1\right\}, \\
		\rho &= \min \left\{ \left ( \frac{|\det A|}{(2^n-1)c} \right )^{1/a}, \left ( \frac{\epsilon}{2} \right )^{1/a}, \frac{1}{2}\right\}.
	\end{align*}
	{\bf Step 4}: Calculate the following quantities:
	\begin{align*}
		&\delta_{1} = \log_R\left ( 1+\frac{\epsilon}{\sqrt{2}R^{b}} \right ),\\ 
		& \delta_{2} =  \log_\rho\left ( 1-\frac{\epsilon}{\sqrt{2}R^{b}} \right ), \\
		&\delta_{3} = \cos^{-1} \left (  1 - \frac{\epsilon^2}{4R^{2b}} \right ).
	\end{align*}
	and take $\delta = \min \{ \delta_1, \delta_2, \delta_3, a\}$.
	
	{\bf Step 5:} For each $i = 1 , \ldots, n$, find a rational number $\beta_i$ such that $\alpha_i - \delta < \beta_i \leq \alpha_i$.
	
	{\bf Output:} Multi-index $\hat\beta = (\beta_1, \ldots, \beta_n)$.
	
	\hrule

	\subsection{Equivalence between the fractional spectrum and its rational approximation}
	\label{subs:rat-equiv}
	
	Consider a matrix $A \in M_n(\mathbb R)$ and a multi-index $\hat\alpha \in (0,1]^n$. 
	Inspired by the definition of the spectral radius of a matrix and the applications of this concept
	in the theory of ordinary differential equations, see e.g., \cite{HINRICHSEN,VanLoan}, 
	we propose the definition
	\begin{align*}
		\delta_{\hat\alpha}^2(A) 
		&:=\inf \left\{\|E\|_2 : E \in M_n(\mathbb C), \sigma_{\hat\alpha}(A+E) \cap \mathbb C_{\geq 0} \neq \emptyset \right\}. 
	\end{align*}
	Suppose further that $\sigma_{\hat\alpha}(A) \subset \mathbb C_{-}$. Then, similar to \cite[Proposition 3.1]{HINRICHSEN}, we have
	\begin{align}\label{bkht}
		\delta_{\hat\alpha}^2(A) 
		&=\inf \left\{\|E\|_2 : E \in M_n(\mathbb C) \text{ and } 
				\sigma_{\hat\alpha}(A+E) \cap \iu \mathbb R \neq \emptyset \right\}\nonumber \\
		&= \inf \left\{\epsilon : \sigma_{\hat\alpha,\epsilon}^2(A) \cap \iu \mathbb R \neq \emptyset \right\} \nonumber\\
		&= \inf \left\{\epsilon : \text{there exists some }z \in \iu \mathbb R
						\text{ such that } \|(z^{\hat\alpha} I - A)^{-1}\|_2^{-1} = \epsilon \right\} \nonumber\\
		&= \min_{\Re(z) = 0}\,\|(z^{\hat\alpha} I - A)^{-1}\|_2^{-1}.
	\end{align}
	
	\begin{remark}\label{rm1}
		From the definition of $\delta_{\hat\alpha}^2(A)$, we see that 
		$\sigma_{\hat\alpha}(A + E) \subset \mathbb C_{-}$
		if $\| E \|_2 < \delta_{\hat\alpha}^2(A)$ for all $E \in M_n(\mathbb C)$. 
	\end{remark}
	
	\begin{remark}\label{rm2}
		If $A \in M_n(\mathbb R)$ and $\sigma_{\hat\alpha}(A) \subset \mathbb C_{-}$, then $\det (z^{\hat\alpha} I - A) \neq 0$ 
		whenever $\Re(z) = 0$. 
		Thus $\|(z^{\hat\alpha} I - A)^{-1}\|_2^{-1} >0$ for all $z \in \mathbb C$ with $\Re(z) = 0$ 
		and $\min_{\Re(z) = 0} \|(z^{\hat\alpha} I - A)^{-1}\|_2^{-1} > 0$, 
		which together with \eqref{bkht} implies that $\delta_{\hat\alpha}^2{(A)} > 0$. 
	\end{remark}
	
	\begin{remark}\label{rm3}
		Assume that $A \in M_n(\mathbb R)$ and $\sigma_{\hat\alpha}(A) \subset \mathbb C_{-}$. 
		Let $\epsilon > 0$ such that $\sigma_{\hat\alpha, \epsilon}^2(A)  \subset \mathbb C_{-}$.
		Then, due to \eqref{eq:pseudo-ii}, we obtain that $ \sigma_{\hat\alpha}(A + E) \subset \sigma_{\hat\alpha, \epsilon}^2(A)  
		\subset \mathbb C_{-}$ for every matrix $E \in M_n(\mathbb C)$ provided that 
		$\| E \|_2 \leq \epsilon$. This implies $\delta_{\hat\alpha}^2(A) \geq \epsilon$. 
		Thus, we have $\delta_{\hat\alpha}^2(A) \geq \sup \left\{\epsilon: \sigma_{\hat\alpha, \epsilon}^2(A)  
		\subset \mathbb C_{-}  \right\}$.
	\end{remark}
	
	\begin{theorem}\label{dl3}
		For a given matrix $A \in M_n(\mathbb R)$ and a multi-index $\hat\alpha \in (0,1]^n$, the following statements are equivalent:
		\begin{itemize}
			\item [(i)] $\sigma_{\hat\alpha}(A) \subset \mathbb C_{-}$;
			\item [(ii)] There is a constant $h_0 >0$ such that for all  $\epsilon \in (0, h_0)$ and all $\epsilon$-rational approximations
				$\hat\beta \in (0,1]^n \cap \mathbb Q^n$ 
				of $\hat\alpha$ associated with $A$, we have
				$\sigma_{\hat\beta}(A) \subset \mathbb C_{-}$ and $\delta_{\hat\beta}^2(A) \geq \epsilon$;
			\item [(iii)] There exists an $\epsilon$-rational approximation $\hat\beta \in (0,1]^n \cap \mathbb Q^n$
				of $\hat\alpha$ associated with $A$
				such that $\sigma_{\hat\beta}(A) \subset \mathbb C_{-}$ and
				$\delta_{\hat\beta}^2(A) \geq \epsilon$. 
		\end{itemize}
	\end{theorem}
	
	\begin{proof} 
		We will first prove that (i) $\Rightarrow$ (ii). Suppose $\sigma_{\hat\alpha}(A) \subset C_{-}$. Then, by Remark \ref{rm2}, we have
		$\delta_{\hat\alpha}^2(A) > 0$. 
		We thus choose $h_0 = \delta_\alpha^2(A) / 2 > 0$. 
		According to Proposition \ref{ttxx}, for all  $0 < \epsilon < h_0$ there exists some 
		$\hat\beta \in (0,1]^n \cap \mathbb Q^n$ 
		which is an $\epsilon$-rational approximation of $\hat\alpha $ associated with $A$.
		Therefore, $\hat\beta$ satisfies the conditions (i)--(iv) of Definition \ref{dnxx} whenever $\epsilon < h_0$. 
		From \eqref{eq:pseudospectra}, we have $ \sigma_{\hat\beta}(A) \subset \sigma_{\hat\beta, \epsilon}^2(A)$. 
		Hence, by Definition \ref{dnxx} (ii), there exists a constant $R$
		such that $ \sigma_{\hat\beta}(A) \cap \{ z \in \mathbb C : |z| > R\} = \emptyset$. 
		Moreover, by Definition \ref{dnxx} (iii), there exists a constant $\rho$ such that 
		$\sigma_{\hat\beta}(A) \cap \{ z \in \mathbb C : |z| <\rho \}= \emptyset$. 
		Consider any $z_0 \in \sigma_{\hat\beta}(A)$. Then $\rho \leq |z| \leq R$ and
		\begin{align}\label{bt6}
			0 = \det(z_0^\beta I - A) = \det (z_0^{\hat\alpha} I - A -( z_0^{\hat\alpha} I - z_0^{\hat\beta} I) ) 
			= \det (z_0^{\hat\alpha} I - (A + E ))
		\end{align}
		with $E = z_0^{\hat\alpha} I - z_0^{\hat\beta} I \in M_n(\mathbb C)$. Thus, $z_0 \in \sigma_{\hat\alpha}( A+ E)$. Furthermore,
		according to Definition \ref{dnxx} (iv), we have $ |z_0^{\alpha_i}- z_0^{\beta_i}| < \epsilon$ for all $i \in N$. Hence,
		\begin{align}\label{bt7}
			\|E\|_2 = \| z_0^{\hat\alpha} I - z_0^{\hat\beta} I  \|_2 = \max_{i \in N}|z_0^{\alpha_i} - z_0^ {\beta_i}| < \epsilon.
		\end{align}
		Since $\epsilon \leq h_0 < \delta_{\hat\alpha}^2 (A), $ by Remark \ref{rm1}
		we see that $\sigma_{\hat\alpha}( A+ E) \subset \mathbb C_{-}$
		which implies $z_0 \in \mathbb C_{-}$.
		Therefore, $\sigma_{\hat\beta}(A) \subset \mathbb C_{-}$. 
		Next, we consider an arbitrary $z_1 \in \sigma_{\hat\beta, \epsilon}^2(A)$. 
		According to \eqref{eq:pseudo-ii}, there exists $E_1 \in M_n(\mathbb C)$ with $\| E_1\|_2 \leq \epsilon$ 
		such that $z_1 \in \sigma_{\hat\beta} (A + E_1)$. 
		This implies that
		\begin{align}\label{bt8}
			0 &= \det (z_1^{\hat\beta} I - (A +E_1)) = \det (z_1^{\hat\alpha} I - (A+E_1)-(z_1^{\hat\alpha} I - z_1^{\hat\beta} I))\nonumber \\
			& = \det (z_1^{\hat\alpha} I - (A+E_1 + E_2))
		\end{align}
		where $E_2 = z_1^{\hat\alpha} I -  z_1^{\hat\beta} I \in M_n(\mathbb C)$. Thus $z_1 \in \sigma_{\hat\alpha}( A+ E_1 + E_2)$. 
		
		On the other hand, since $z_1 \in \sigma_{\hat\beta, \epsilon}^2(A)$,  Definition \ref{dnxx} (ii) implies $|z_1| \leq R$, and
		by Definition \ref{dnxx} (iv), we have $|z_1^{\alpha_i} - z_1^{\beta_i}| < \epsilon$ for all $i \in N$. Hence,
		\begin{align}\label{bt9}
			\|E_2\|_2 =  \|z_1^{\hat\alpha} I - z_1^{\hat\beta} I\|_2 = \max_{i \in N}|z_1^{\alpha_i} - z_1^{\beta_i}| < \epsilon.
		\end{align}
		So, $\| E_1 + E_2 \|_2 \leq \| E_1\|_2 + \| E_2\|_2 < \epsilon + \epsilon \leq 2h_0 = \delta_{\hat\alpha}^2{A}$. 
		Consequently, by Remark \ref{rm1}, we have
		$\sigma_{\hat\alpha} (A + E_1 + E_2) \subset \mathbb C_{-}$,
		and it follows that $z_1 \in \mathbb C_{-}$. Hence, $\sigma_{\hat \beta, \epsilon}^2(A) \subset \mathbb C_{-}$.
		By Remark \ref{rm3}, we have $\delta_{\hat\beta}^2(A) \geq \epsilon$. Thus, we have proved (i) $\Rightarrow$ (ii).
		
		(ii) $\Rightarrow$ (iii) is obvious because of Proposition \ref{ttxx}.
		
		Finally, we will prove (iii) $\Rightarrow$ (i). Suppose that $\hat\beta$ is an $\epsilon$-rational approximation 
		of $\hat\alpha$ associated with $A$ such that $\sigma_{\hat\beta}(A) \subset \mathbb C_{-}$ and 
		$\delta_{\hat\beta}^2(A) \geq \epsilon$. 
		Let $z_2 \in \sigma_{\hat\alpha}(A)$ be arbitrary. Then,
		\begin{align}\label{bt10}
			0 = \det (z_2^{\hat\alpha} I - A) 
			= \det (z_2^{\hat\beta} I - A -(z_2^{\hat\beta} I - z_2^{\hat\alpha} I)) = \det (z_2^{\hat\beta} I - (A+E_3))
		\end{align}
		where $E_3 = z_2^{\hat\beta} I -  z_2^{\hat\alpha} I \in M_n(\mathbb C)$. 
		Thus, $z_2 \in \sigma_{\hat\beta}( A+ E_3)$. On the other hand, by \eqref{eq:pseudospectra}, 
		we have $ \sigma_{\hat\alpha}(A) \subset \sigma_{\hat\alpha, \epsilon}^2(A)$.
		Since $\beta$ is an $\epsilon$-rational approximation of $\hat\alpha$ associated with $A$, according to Definition \ref{dnxx} (ii) and (iii), 
		there exist constants $\rho$ and $R$ with $0 < \rho < 1 < R$ such that $\rho \leq |z_2| \leq R$. 
		Then, by Definition \ref{dnxx} (iv), we have $|z_2^{\alpha_i} - z_2^{\beta_i}| < \epsilon$ for all $i \in N$ which implies that
		\begin{align}\label{bt11}
			\|E_3\|_2 = \|(z_2^{\hat\beta} - z_2^{\hat\alpha}) I \|_2 = \max_{i \in N}|z_2^{\alpha_i} - z_2^{\beta_i}| < \epsilon.
		\end{align}
		Since $\epsilon < \delta_{\hat\beta}^2(A)$, by Remark \ref{rm1} we see that $\sigma_{\hat\beta}( A + E_3) \subset \mathbb C_{- }$. 
		Hence, $z_2 \in \mathbb C_{-}$ and therefore, since $z_2$ was an arbitrary element of $\sigma_{\hat\alpha} (A)$,
		we can conclude that
		 $\sigma_{\hat\alpha} (A) \subset \mathbb C_{-}$. Thus, we have completed the proof that (iii) $\Rightarrow$ (i) 
		 and hence also the proof of the complete theorem.
	\end{proof}
	As discussed above, we have given a criterion for testing whether the fractional spectrum of a matrix is lying in the open left half 
	of the complex plane. This criterion is based on rational approximations of the fractional spectrum. 
	An important step in this process is to estimate the positive lower bounds of 
	$\delta_{\hat\alpha}^2(A)$ to find a suitable approximation. Now we will discuss in detail a case where the lower bound estimate for 
	$\delta_{\hat\alpha}^2(A)$ is explicitly specified and thereby establish an algorithm that checks whether
	$\sigma_{\hat\alpha}(A)$ is in $\mathbb C_{-} $.
	
	\begin{proposition}\label{md1}
		Let $A \in M_n(\mathbb R)$ and $\hat\alpha \in (0,1]^n$. In addition, suppose that $\sigma_{\hat\alpha}(A) \subset \mathbb C_{-}$ 
		and $\lambda_{\min}(-(A + A^{\mathrm T})) > 0$, where $\lambda_{\min}(-(A + A^{\mathrm T}))$ 
		is the smallest eigenvalue of the matrix $-(A + A^{\mathrm T})$. 
		Then, $\delta_{\hat\alpha}^2(A )\geq \frac{1}{2}\lambda_{\min}(-(A + A^{\mathrm T})) > 0$.
	\end{proposition}
	
	\begin{proof}
		In view of $\sigma_{\hat\alpha}(A) \subset \mathbb C_{-}$, by \eqref{bkht} and Remark \ref{rm2} we have
		$$
			\delta_{\hat\alpha}^2(A )= \min_{\Re(z) = 0} \|\left ( z^{\hat\alpha} I - A \right )^{-1 }\|_2^{-1} > 0 .
		$$
		This implies that
		$$
			(\delta_{\hat\alpha}^2(A))^{-1} = \max_{\Re(z) = 0} \|\left ( z^{\hat\alpha} I - A \right )^{-1}\|_2 > 0 .
		$$ 
		From this relation we deduce that there exists some $\omega_0 \in \mathbb R$ with 
		$\max_{\Re(z) = 0} \|\left ( z^{\hat\alpha} I - A \right )^{-1}\|_2 =$\linebreak
		$\|\left ( (\iu \omega_0)^{\hat\alpha} I - A \right )^{-1}\|_2$.
		Therefore, there exists $u_0 \in \mathbb C^n$ with $\|u_0\|_2 = 1$ such that 
		$\|\left( (\iu\omega_0)^{\hat\alpha} I - A \right ) ^{-1}\|_2 = \|\left ( (\iu \omega_0)^{\hat\alpha} I - A \right )^{-1} u_0\|$. 
		Using the notation $x = \left ( (\iu\omega_0)^{\hat\alpha} I - A \right )^{-1} u_0$, we
		see that $\|x\| = (\delta_{\hat\alpha}^2(A))^{-1} > 0$, so $x \neq 0$. 
		Applying the min-max theorem to the Hermitian matrix $-(A + A^{\mathrm T}) = -(A + A^*)$
		(note that $A$ is a real matrix by assumption), we have
		\begin{align*}
			\lambda_{\min}(-(A + A^*))\|x\|_2^2 & \leq \left< (-(A + A^*)x,x\right>\\
			&\leq \left< \left ( (\iu \omega_0)^{\hat\alpha} I  - A + ((\iu \omega_0)^{\hat\alpha} I - A)^* 
				- 2 \Re\left ( (\iu \omega_0)^{\hat\alpha} I \right )  \right )x, x\right>
		\end{align*}
		where $\Re\left ( (\iu\omega_0)^{\hat\alpha} I \right ) = 
		\diag(|\omega_0^{\alpha_1}| \cos\frac{\alpha_1\pi}{2}, \ldots,
		|\omega_0^{\alpha_n}| \cos\frac{\alpha_n\pi}{2})$.
		Thus, 
		\begin{align}
		\label{bt12}
			\lambda_{\min}(-(A + A^*))\|x\|_2^2 +2\left<  \Re\left ( (\iu \omega_0)^\alpha I \right )   x  ,x\right> 
				&\leq \left< \left ( (i\omega_0)^{\hat\alpha} I  - A + ((\iu \omega_0)^{\hat\alpha} I - A)^*   \right )x, x\right> \nonumber \\
				&= 2\Re \left ( \left< \left ( (\iu \omega_0)^{\hat\alpha} I - A \right )x,x\right> \right ) \nonumber \\
				&= 2\Re\left< u_0,x\right> 
					\leq 2|\Re\left< u_0,x\right>| 
					\leq 2 |\left< u_0,x\right>| \nonumber \\
				& \leq 2\|u_0\|_2 \|x\|_2 = 2\|x\|_2.
		\end{align}
		On the other hand, 
		\begin{align}
			\label{bt13}
			2\left<  \Re\left ( (\iu \omega_0)^{\hat\alpha} I \right )  x ,x\right> 
			&= \left< \left ( |\omega_0|^{\alpha_1}\cos\frac{\alpha_1\pi}{2}x_1, \ldots, |\omega_0|^{\alpha_n}
					\cos\frac{\alpha_n\pi}{2}x_n\right )^{\mathrm T}, x \right> \nonumber \\
			&= \sum_{i=1}^{n}|\omega_0|^{\alpha_i}\cos \frac{\alpha_i\pi}{2}|x_i|^2 \geq 0.
		\end{align}
		Using \eqref{bt12} and \eqref{bt13}, we see 
		\begin{align*}
			\lambda_{\min}(-(A + A^*))||x||_2^2 \leq 2\|x\|_2,
		\end{align*}
		which implies that $\|x\|_2^{-1} \geq \frac{1}{2}\lambda_{\min}(-(A + A^*))$. 
		Recalling once again that $A \in M_n(\mathbb R)$, we conclude 
		$\delta_{\hat\alpha}^2(A)  \geq\frac{1}{2}\lambda_{\min}(-(A + A^{\mathrm T}))$. 
		The proof is complete. 
	\end{proof}
	
	The arguments of our proofs allow us to formulate an algorithm to check,
	for matrices $A$ satisfying the condition $\lambda_{\min}(-(A + A^{\mathrm T})) > 0$, whether $\sigma_{ \hat\alpha}(A)$
	lies in the open left half of the complex plane:
	
	{\bf Algorithm 2}
	
	\hrule\vskip-0.5em
	{\bf Input}: Matrix $A$ satisfying $\lambda_{\min}(-(A + A^{\mathrm T})) > 0$, and a multi-index $\hat\alpha$.
	
	{\bf Step 1}: Calculate $\lambda_{\min}(-(A + A^{\mathrm T}))$ and put $h_0 =\frac{1}{2}\lambda_{\min}(-(A + A^{\mathrm T}))$.

	{\bf Step 2}: Apply Algorithm 1 using the matrix $A$, the multi-index $\hat\alpha$ and $\epsilon = h_0$ 
		as input data to find $\hat\beta$ which is an $\epsilon$-rational approximation of $\hat\alpha$ associated with $A$.

	{\bf Step 3}: Check if $\sigma_{\hat\beta}(A)$ lies in the open left half of the complex plane. 
		If $\sigma_{\hat\beta}(A) \subset \mathbb C_{-}$, 
		we conclude that $\sigma_{\hat \alpha}(A) \subset \mathbb C_{-}$. 
		If $\sigma_{\hat \beta}(A) \nsubseteq \mathbb C_{-}$, we conclude that $\sigma_{\hat \alpha}(A) \nsubseteq \mathbb C_{-}$.
	
	{\bf Output}: The result of Step 3, i.e.\ the information whether or not $\sigma_{\hat \alpha}(A)$ 
		lies in the open left half of the complex plane.
	
	\hrule\bigskip
	
	\begin{example}\label{example2}
		To illustrate the proposed algorithms, we consider the system 
		\begin{equation}\label{Eq27}
			^C D^{\hat\alpha}_{0^+} x(t) = Ax(t),\; t > 0
		\end{equation}
		with 
		$$ 
			 A = \begin{pmatrix}
				-0.5&-0.2  &-0.15  &0.25  \\
				0.15&-0.4  &0.2  &-0.15  \\
				0.25& 0.15  &-0.6  &0.3  \\
				0.2& -0.1 &-0.1  &-0.3  \\
			\end{pmatrix}
		$$ 
		(as in Example \ref{vd1}) and the multi-index $\hat\alpha = ( \alpha_1, \alpha_2, \alpha_3, \alpha_4) = 
		\left ( \frac{128}{71 \sqrt{13}}, \frac{64}{71 \sqrt{13}}, 
		\frac{90}{47 \sqrt{33}}, \frac{45}{47 \sqrt{33}} \right )$.
		By direct calculations we have $\lambda_{\min}(-(A + A^{\mathrm T})) \approx 0.204$. 
		We may therefore set $h_0 = 0.1$ 
		and find the $0.1$-rational approximation $\hat\beta$ of $\hat\alpha$ associated 
		with $A$ using Algorithm 1 as follows:
		We have $a = \frac{1}{2}\min_{i \in N}\{\alpha_i\} = \frac{45}{94\sqrt{33}}$, 
		$b = \max_{i \in N}\{\alpha_i\} = \frac{128}{71\sqrt{13}}$ and $c =1$.
		By simple calculations, we get 
		\begin{align*}
			R &= (1.75 + 0.1)^{1/a} \approx 1606.922,  \\ 
			\rho &= \left ( \frac{3759}{80000 \times 15} \right )^{1/a} \approx 8.94 \times10^{-31}, \\
			\delta_1 & = \log_R \left ( 1 + \frac{\epsilon}{\sqrt{2}R^b} \right ) \approx 0.000239, \\
			\delta_2 & = \log_\rho \left ( 1 - \frac{\epsilon}{\sqrt{2}R^b} \right ) \approx 0.0000255, \\
			\delta_3 & = \frac 1 \pi \cos^{-1} \left ( 1 - \frac{\epsilon^2}{4R^{2b}}\right ) \approx 0.000561. 
		\end{align*}
		This implies that $\delta = \delta_2 = 0.0000255$. 
		Furthermore, $0.4995 \approx \alpha_1 - \delta < \beta_1 < \alpha_1 
		\approx 0.50001$. Hence, we can take $\beta_1 = 1/2$.  
		Similarly, we have $\beta_2 = 1/4$, $\beta_3 = 1/3$ and $\beta_4 = 1/6$
		which shows that $\hat\beta = \left (\frac{1}{2}, \frac{1}{4}, \frac{1}{3}, \frac{1}{6} \right )$ is a 
		$0.1$-rational approximation of $\hat\alpha$ associated with $A$. 
		From Example \ref{vd1}, we see that $\sigma_{\hat\beta}(A) \subset \mathbb C_{-}$ 
		and thus $\sigma_{\hat\alpha} (A) \subset \mathbb C_{-}$. 
		A plot of the corresponding solution function graphs is visually undistinguishable from the 
		plot shown in Figure \ref{refhinh1}, therefore we do not show this explicitly here. But clearly,
		this indicates the asymptotic stability in the case discussed here too. This effect could have been 
		expected because the difference between the system considered here and the system of 
		Example \ref{vd1} above is only a tiny change in the orders of the differential operators,
		and standard theoretical results \cite[Theorem 6.22]{Kai} show that---unless the generalized 
		eigenvalues of the original system
		had been so close to the boundary of the stability region that this change had made them 
		move to the other side of the boundary, which is not the case here---such small changes only lead to
		correspondingly small changes in the solutions.
		
	\end{example}

\section{Asymptotic behavior of solutions to incommen\-surate frac\-tional-order nonlinear systems}
\label{s5}

	Based on the developments above, we can now state some results about the stability of 
	fractional multi-order differential systems. We will begin with a discussion of the case of a linear system and
	deal with the nonlinear case afterwards.
	
	\subsection{Inhomogeneous linear systems}
	\label{st5}
	
	Consider the inhomogeneous linear mixed-order system
	\begin{align}\label{bt3.1}
		^C D^{\hat{\alpha}}_{0^+}x(t) = Ax(t) +f(t),\;t>0
	\end{align}
	with the initial condition
	\begin{align}\label{dkd3}
		x(0)= x^0 \in \mathbb R^n,
	\end{align}
	where $\hat\alpha = (\alpha_1, \alpha_2, \ldots \alpha_n) \in (0, 1]^n$, 
	$A\in \mathbb R^{n\times n}$ and $f: [0, \infty) \rightarrow \mathbb R^n$
	is continuous and exponentially bounded, that is, there exist constants $M, \gamma > 0$ 
	such that $\|f(t)\| \leq Me^{\gamma t}$ for all $t \in [0, \infty)$.
	We will first establish a variation of constants formula for the problem \eqref{bt3.1}-\eqref{dkd3}. 
	To this end, we may generalize the approach described in \cite[Subsection 2.2]{TuanThai_2022} for the
	case $n=2$, i.e.\ we 
	take the Laplace transform on both sides of the system \eqref{bt3.1} and incorporate the
	initial condition \eqref{dkd3} to get the algebraic equation
	\begin{align}
		(s^{\hat\alpha} I ) X(s) - (s^{\hat\alpha -1}I) x^0 = A X(s)+F(s),  
	\end{align}
	where $s^{\hat\alpha}I = \diag(s^{\alpha_1}, \ldots, s^{\alpha_n})$,
	$s^{\hat\alpha -1}I = \diag(s^{\alpha_1-1}, \ldots, s^{\alpha_n-1})$ and $X(\cdot)$ and $F(\cdot)$ are the 
	Laplace transforms of $x(\cdot)$ and $f(\cdot)$, respectively. Thus,
	\begin{align}\label{ptds}
		X(s) = (s^{\hat\alpha} I - A)^{-1} \left( (s^{\alpha_1 - 1}x^0_1,\ldots,s^{\alpha_n-1}x^0_n)^{\mathrm T}
					+ F(s) \right).  
	\end{align}
	Since $(s^{\hat\alpha} I - A)^{-1} =\frac{1}{Q(s)} \left ( (-1)^{i+j}\Delta^A_{ij}(s) \right )_{n \times n}$, 
	where $Q(s) = \det(s^{\hat\alpha} I- A)$ and $\Delta^A_{ij}(s)$ is the determinant of the matrix 
	obtained from the matrix 
	$s^{\hat\alpha} I - A$ by removing the $j$-th row and the $i$-th column, for each $i \in N$ we have
	\begin{align}\label{ctbdn}
		X_i(s) = \sum_{j=1}^{n}\frac{1}{Q(s)}\left ( (-1)^{i+j}\Delta^A_{ij}(s)\right)s^{\alpha_j-1}x^0_j
				+  \sum_{j=1}^{n}\frac{1}{Q(s)}\left ( (-1)^{i+j}\Delta^A_{ij}(s)\right)F_j(s).
	\end{align}
	Next, we will explicitly calculate the terms $\Delta^A_{ij}(s)$. For $i = j$, we put 
	$$
		\widetilde{\hat{\alpha}^i} = (\alpha_1, \ldots, \alpha_{i-1}, \alpha_{i+1},\ldots, \alpha_n)
	$$ 
	and designate by $A_{(i;i)}$ the matrix obtained from the matrix $A$ by removing 
	the $i$-th row and $i$-th column. Then,
	$$ 
		\Delta^A_{ii}(s) = \det (s^{\widetilde{\hat{\alpha}^i}} I - A_{(i;i)}). 
	$$
	Proceeding much as in Section \ref{s3}, we obtain
	\begin{align}\label{dt1}
		\Delta^A_{ii}(s) 
		= \sum_{\eta \in \mathcal B^{n-1}} c^{(i;i)}_\eta s^{\left< \widetilde{\hat{\alpha}^i},\eta \right>}
	\end{align}
	where $\mathcal B^{n-1} = \{0, 1 \}^{n-1}$
	and, for every $\eta \in \mathcal B^{n-1}$, the $c^{(i;i)}_\eta$ are constants that depend 
	only on the matrix $A_{(i;i)}$. Put
	$$ 
		\hat{\mathcal B}^{n}_i = \{\xi \in \mathcal B^n : \xi_i = 0 \}
	$$
	where $\mathcal B^n = \{0, 1\}^n$ as in the proof of Proposition \ref{ttxx}.
	We see that 
	$$
		\left\{\left< \widetilde{\hat{\alpha}^i},\eta \right>: \eta \in \mathcal B^{n-1}\right\} 
		= \left\{\left<\hat\alpha, \xi \right>: \xi \in \hat{\mathcal B}^{n}_i \right\}, 
	$$ 
	and thus 
	$$ 
		\Delta^A_{ii} = \sum_{\xi \in \hat{\mathcal B}^n_i}c^{(i;i)}_\xi s^{\left< \hat\alpha, \xi\right>}. 
	$$
	Let 
	$$
		\widetilde{\mathcal B}^n_i =\left\{ \nu \in \mathcal B^n : \nu_i = 1  \right\}, 
	$$
	then 
	\begin{align}\label{ct1}
		\Delta^A_{ii}s^{\alpha_i-1} 
		&= \sum_{\xi \in \hat{\mathcal B}^n_i}c^{(i;i)}_
			\xi s^{\left< \hat\alpha, \xi\right>}s^{\alpha_i -1} 
		= \sum_{\xi \in \hat{\mathcal B}^n_i}c^{(i;i)}_
			\xi s^{\left< \hat\alpha, \xi\right>+\alpha_i -1} 
		 =\sum_{\nu \in \widetilde{\mathcal B}^n_i}c^{(i;i)}_
			\nu s^{\left< \hat\alpha, \nu\right> -1}.
	\end{align}
	The last equality above is obtained because
	$$
		\left\{\left<\hat\alpha, \xi \right> + \alpha_i : \xi \in \hat{\mathcal B}^n_i \right\}
		= \left\{\left< \hat\alpha, \nu\right> : \nu \in  \widetilde{\mathcal B}^n_i\right\} . 
	$$
	Next, for $i, j \in N$ with $i \neq j$, we put
	\begin{align}
		\widetilde{\alpha^i_j} = 
			\begin{cases}
				(\alpha_1,\ldots,\alpha_{i-1}, \alpha_{i+1},\ldots, \alpha_{j-1}, 0, \alpha_{j+1},\ldots,\alpha_n)
					 & \text{ if } i < j, \\
				(\alpha_1,\ldots,\alpha_{j-1}, 0, \alpha_{j+1},\ldots, \alpha_{i-1}, \alpha_{i+1},\ldots,\alpha_n)
					& \text{ if } i > j, 
			\end{cases}
	\end{align} 
	and $\hat A = A + {\bf 1}_{ij}$ where ${\bf 1}_{ij}$ is the $n \times n$ matrix whose element 
	at the $i$-th row and the $j$-th column is 1 while all other entries are 0. Then,
	$$
		\Delta^A_{ij} = \det(s^{\widetilde{\alpha^i_j}}I - \hat A_{(j;i)}), 
	$$
	where $\hat A_{(j,i)}$ is the matrix obtained from the matrix $\hat A$ 
	by omitting the $j$-th row and the $i$-th column. Thus, 
	\begin{align}
		\Delta^A_{ij}(s) = \sum_{\eta \in \mathcal B^{n-1}} c^{(i;j)}_\eta s^{\left< \widetilde{\hat{\alpha}^i_j},\eta \right>}, 
	\end{align}
	where,  for every $\eta \in \mathcal B^{n-1}$, the $c^{(i;j)}_\eta$ are constants that depend only on the 
	matrix $\hat A_{(j;i)}$. Put 
	$$ 
		\hat{\mathcal B}^n_{(i;j)} = \left\{\xi \in \mathcal B^n : \xi_i = \xi_j = 0 \right\} . 
	$$
	Then,
	$$ 
		\left\{ \left< \widetilde{\hat{\alpha}^i_j},\eta\right> : \eta \in \mathcal B^{n-1} \right\} 
			= \left\{\left< \hat\alpha, \xi\right> : \xi \in \hat{\mathcal B}^n_{(i;j)} \right\}. 
	$$
	Thus, 
	\begin{align}
		\Delta^A_{ij}(s) 
		= \sum_{\xi \in \hat{\mathcal B}^n_{(i;j)}} c^{(i;j)}_\xi s^{\left< \hat\alpha, \xi \right>}. 
	\end{align}
	Similarly, writing
	$$
		\widetilde{\mathcal B^n_{i;j}} = \left\{ \zeta \in \mathcal B^n : \zeta_i = 0, \zeta_j = 1\right\}, 
	$$
	we obtain 
	\begin{align}\label{ct2}
		\Delta^A_{ij}s^{\alpha_j-1} 
		&= \sum_{\xi \in \hat{\mathcal B}^n_{(i;j)}} 
			c^{(i;j)}_\xi s^{\left< \hat\alpha, \xi \right>}s^{\alpha_j -1} 
		= \sum_{\xi \in \hat{\mathcal B}^n_{(i;j)}}c^{(i;j)}_
			\xi s^{\left< \hat\alpha, \xi\right>+\alpha_j -1}
		=\sum_{\zeta \in \widetilde{\mathcal B^n_{i;j}}}c^{(i;j)}_
			\zeta s^{\left< \hat\alpha, \zeta\right> -1}.
	\end{align}
	The last equality in \eqref{ct2} is obtained by
	$$
		\left\{\left< \hat\alpha, \xi\right>+\alpha_j: \xi \in \hat{\mathcal B}^n_{(i;j)} \right\} 
		= \left\{\left< \hat\alpha, \zeta\right>:\zeta \in \widetilde{\mathcal B^n_{i;j}} \right\}. 
	$$
	Taking
	$$ 
		\mathcal M_i = \left\{\lambda : \lambda = \left< \hat\alpha, \nu\right>,\nu \in \widetilde{\mathcal B}^n_i
					\,\, \text{ or } \,\, 
					\lambda = \left< \hat\alpha, \zeta\right>, \zeta \in \widetilde{\mathcal B^n_{i;j}}  \right\}, 
	$$
	we have, for all $i \in N$,
	\begin{align}\label{ct3}
		\sum_{j=1}^{n} \frac{1}{Q(s)} \left ( (-1)^{i+j} \Delta^A_{ij}(s) \right) s^{\alpha_j-1} x^0_j 
		= \sum_{\lambda \in \mathcal M_i} c^i_\lambda \frac{s^\lambda}{s Q(s)}
	\end{align}
	with certain uniquely determined constants $c^i_\lambda \in \mathbb R$.
	In much the same way, setting
	$$
		\mathcal N_i = \left\{ \beta : \beta =\left<\hat\alpha, \xi \right>, \xi \in \hat{\mathcal B}^n_i \,\, \text{ or }\,\, 
		\beta = \left< \hat\alpha, \eta \right>, \eta \in \hat{\mathcal B}^n_{(i;j)} \right\}, 
	$$
	we have
	\begin{align}\label{ct5}
		\sum_{j=1}^{n} \frac{1}{Q(s)} \left ( (-1)^{i+j} \Delta^A_{ij}(s) \right) F_j(s)
		= \sum_{\beta \in \mathcal N_i} c_\beta^i \frac{s^\beta}{Q(s)} F(s)
	\end{align}
 	for every $i \in N$, where once again the real constants $c_\beta^i$ are uniquely determined.
	From \eqref{ctbdn}, \eqref{ct3} and \eqref{ct5}, we conclude
	\begin{align}
		X_i(s) = \sum_{\lambda \in \mathcal M_i}c^i_\lambda\frac{s^\lambda}{sQ(s)} 
		+ \sum_{\beta \in \mathcal N_i}c_\beta^i\frac{s^\beta}{Q(s)}F(s).
	\end{align}
	Thus, defining
	\begin{align}\label{R}
		R_i^\lambda(t) = \mathcal L^{-1}\left\{\frac{s^\lambda}{sQ(s)} \right\}
		\text{ for } \lambda \in \mathcal M_i,
	\end{align}
	and
	\begin{align}\label{S}
		S_i^\beta(t) = \mathcal L^{-1}\left\{\frac{s^\beta}{Q(s)} \right\}
		\text{ for } \beta \in \mathcal N_i,
	\end{align}
	we get the variation of constants formula for the problem \eqref{bt3.1}-\eqref{dkd3} as follows: 
	\begin{align}\label{ctbths1}
		x_i(t)
			= \sum_{\lambda \in \mathcal M_i}c^i_\lambda R_i^\lambda(t) + 
				\sum_{\beta \in \mathcal N_i}c_\beta^i (S_i^\beta*f_i) (t), \quad i=1, 2,\ldots, n.
	\end{align}
	
	To determine the asymptotic behaviour of the functions $x_i$ for $i \in N$---and hence the
	stability properties of the differential equation \eqref{bt3.1}---from eq.~\eqref{ctbths1},
	we need to obtain information about the asymptotic behaviour of the functions $R^\lambda_i$ and $S^\beta_i$.
	For this purpose, we can argue in exactly the same way as in \cite[Lemma 8]{TuanThai_2022}. 
	This leads us to the following result.
	
	\begin{lemma}
		\label{lemma:asymptotics-r-s}
		Let $\hat\alpha \in (0,1]^n$. Put 
		$\nu = \min \{\alpha_1, \alpha_2, \ldots, \alpha_n\}$. Assume that $\sigma_{\hat\alpha}(A)$  
		lies in the open left half of the complex plane. 
		Then, for all $i \in N$, $\lambda \in \mathcal M_i$ and $\beta \in \mathcal N_i$, 
		we have the following asymptotic behaviour:
		\begin{align}
			\label{eq:asymp1}
			R^\lambda_i(t) &= O(t^{-\nu}) \text{ as } t \to \infty, \\
			\label{eq:asymp2}
			S^\beta_i(t)      & = O(t^{-\nu-1}) \text{ as } t \to \infty, \\
			\label{eq:asymp3}
			S^\beta_i(t)      &= O(t^{\nu-1}) \text{ as } t \to 0.
		\end{align}
		Furthermore,
		$$
			\int_0^\infty |S^\beta_i(t)| \du t < \infty. 
		$$
	\end{lemma}
	
	Next, we apply the estimates of Lemma \ref{lemma:asymptotics-r-s}
	to the derive an intermediate result
	that will in the next step allow us to describe the behaviour of the terms in the
	second sum on the right-hand side of formula \eqref{ctbths1}, i.e.\ the asymptotic behaviour of the convolution 
	of $S^{\beta}_i$ with continuous functions. This result is a direct generalization of
	\cite[Theorem 3]{TuanThai_2022} and can be proved in the same manner.
	
	\begin{theorem}
		\label{thm:asymp}
		Let $\hat\alpha \in (0,1]^n$, $\nu = \min \{\alpha_1, \alpha_2, \ldots, \alpha_n\}$ 
		and $\beta \in \mathcal N_i$ for some $i \in N$. 
		For a given continuous function $g: [0,\infty) \rightarrow \mathbb R$, we set  
		$$
			F^{\beta}_i(t) := (S^\beta_i * g) (t) = \int_0^t S^\beta_i(t-s) g(s) \du s. 
		$$
		Suppose that $\sigma_{\hat\alpha}(A)\subset \mathbb C_{-}$. Then,  the following statements are true.
		\begin{itemize}
		\item [(i)] If $g$ is bounded then $F^{\beta}_i$ is also bounded.
		\item [(ii)] If $\lim_{t \to \infty} g(t) = 0$ then $\lim_{t \to \infty}F^{\beta}_i(t) = 0$.
		\item [(iii)] If there exists some $\eta > 0$ such that $g(t) = O(t^{-\eta})$ as $t \to \infty$, 
				then $F^{\beta}_i(t) = O(t^{-\mu})$ as $t \to \infty$ where $\mu = \min\{ \nu, \eta\}$.
		\end{itemize}
	\end{theorem}
	
	From the above assertions, we obtain the following results on the asymptotic behavior of 
	solutions to the inhomogeneous linear mixed order system \eqref{bt3.1}.
	
	\begin{theorem}\label{main_result_1}
		Consider the initial value problem \eqref{bt3.1}-\eqref{dkd3} with
		$\hat\alpha \in (0,1]^n$. Set $\nu = \min \{\alpha_1, \alpha_2, \ldots, \alpha_n\}$ and assume that 
		$\sigma_{\hat\alpha}(A)\subset \mathbb C_{-}$. Then the following assertions hold. 
		\begin{itemize}
		\item [(i)] If $f$ is bounded then the solution of the initial value problem is also bounded,
			no matter how the initial value vector $x^0$ in \eqref{dkd3} is chosen.
		\item [(ii)] If $\lim_{t \to \infty} f(t) = 0$ then the solutions of \eqref{bt3.1} converge to $0$ 
			as $t \to \infty$
			for any choice of the initial value vector $x^0$.
		\item [(iii)] If there is some $\eta > 0$ such that $\|f(t) \| = O(t^{-\eta})$ as $t \to \infty$
			then, for any initial value vector $x^0$, the solution $x(\cdot)$ of \eqref{bt3.1} 
			satisfies $\|x(t)\| = O(t^{-\mu}) $ as $t \to \infty$, where $\mu = \min\{\nu, \eta\}$.  
		\end{itemize}
	\end{theorem}

	\begin{proof}
		This is a straightforward generalization of \cite[Theorem 4]{TuanThai_2022} that can be shown
		in an analog manner, using our Theorem \ref{thm:asymp}.
	\end{proof}
	
	\subsection{Nonlinear Systems}

	Finally, we consider the autonomous incommensurate fractional order nonlinear system
	\begin{align}\label{bt4}
		^C D^{\hat\alpha}_{0^+}x(t) & = Ax(t) +f(x(t)),\quad t > 0, \\
		x(0) & = x^0 \in \Omega\subset\mathbb R^n,\label{ic4}
	\end{align}
	where $\hat \alpha=(\alpha_1,\alpha_2, \ldots, \alpha_n)\in (0,1]^n$,  
	$A = \left ( a_{ij} \right ) \in \mathbb R^{n \times n}$, 
	$\Omega$ is an open subset of $\mathbb R^n$ with $0\in \Omega$, $f:\Omega\rightarrow \mathbb R^n$ 
	is locally Lipschitz continuous at the origin such that
	$f(0)=0$ and $\lim_{r\to 0} l_f (r)=0$ with 
	\[
		l_f(r) := \sup_{x,y\in B(0,r),\;x\ne y}\frac{\|f(x)-f(y)\|}{\|x-y\|}.
	\]
	Putting $g(t)=f(x(t))$ and repeating the arguments as in Subsection \ref{st5}, 
	we get the representation of the solution of the problem \eqref{bt4}
	\begin{align}
		\label{ctbths1.1}
		x_i(t)= \sum_{\lambda \in \mathcal M_i}c^i_\lambda R_i^\lambda(t) + 
		\sum_{\beta \in \mathcal N_i}c_\beta^i (S_i^\beta*f_i) (x(t)),
		\quad i=1, 2,\ldots,n.
	\end{align}
	We recall here the Mittag-Leffler stability definition that was introduced in \cite[Definition 2]{TuanThai_2022}.
	\begin{definition}
		\label{def:mlstable}
		The trivial solution of \eqref{bt4} is Mittag-Leffler stable if there exist positive
		constants $\gamma,m$ and $\delta$ such that for any initial condition $x^0\in B(0,\delta)$, 
		the solution $\varphi(\cdot, x^0)$ of the initial value problem \eqref{bt4}-\eqref{ic4} 
		exists globally on the interval $[0,\infty)$ and
		\[
			\max \left \{ \sup_{t\in [0,1]} \|\varphi(t, x^0)\| , 
							\sup_{t \ge 1} t^\gamma \| \varphi(t, x^0) \| \right \} 
			\le m.
		\]
	\end{definition}
	
	By the same approach as in \cite[ Theorem 5]{TuanThai_2022}, we obtain the Mittag-Leffler stability 
	of the trivial solution of \eqref{bt4}:
	
	\begin{theorem}\label{main_result_2}
		Consider the system \eqref{bt4}. Assume that $\sigma_{\hat\alpha}(A)\subset \mathbb C_{-}$. 
		Then the trivial solution of the system of equations \eqref{bt4} is Mittag-Leffler stable. 
		More precisely, there exist constants $\delta, \epsilon > 0$ such that the unique global solution 
		$\varphi(\cdot,x^0)$ 
		of the initial value problem \eqref{bt4}--\eqref{ic4} satisfies the estimate 
		$\sup_{t \geq 1} t^\nu \| \varphi(t,x^0)\| \leq \epsilon$ 
		with $\nu = \min \{\alpha_1, \alpha_2, \ldots, \alpha_n\}$ provided that $\|x^0\| < \delta$.
	\end{theorem}

\section{Examples}
\label{s6}

	This section is devoted to introducing some examples to illustrate the 
	validity of the two main results obtained in Section 5.

	\begin{example}
		\label{ex:61}
		We consider the system 
		\begin{align}
			^C D^{\hat\alpha}_{0^+} x(t)& = Ax(t)+f(t),\; t > 0,\label{Eq54} \\
			x(0)&=x^0\in \R^4,\label{incond_1}
		\end{align}
		where
		$$ 
			A = \begin{pmatrix}
				-0.5&-0.2  &-0.15  &0.25  \\
				0.15&-0.4  &0.2  &-0.15  \\
				0.25& 0.15  &-0.6  &0.3  \\
				0.2& -0.1 &-0.1  &-0.3  
				\end{pmatrix},
		$$
		$\hat\alpha = ( \alpha_1, \alpha_2, \alpha_3, \alpha_4) = 
			\left ( \frac{128}{71 \sqrt{13}}, \frac{64}{71 \sqrt{13}}, \frac{90}{47 \sqrt{33}}, 
			\frac{45}{47 \sqrt{33}} \right )$
		and $f(t) = (f_1(t), f_2(t), f_3(t), f_4(t))^{\mathrm T}$ with $f_i(t) =(1 + t^i)^{-1}$, $i = 1, 2, 3, 4$. 
		As shown in Example \ref{example2}, we see that $\sigma_{\hat\alpha}(A) \subset \mathbb C_{-}$.
		Moreover, $\|f(t)\| = O(t^{-1})$ as $t \to \infty$. Due to Theorem \ref{main_result_1}, 
		the system \eqref{Eq54} is globally attractive and the solution $\varphi(\cdot,x^0)$ 
		satisfies $\|\varphi(t,x^0)\| = O(t^{-\frac{45}{47\sqrt{33}}})$ as $t \to \infty$ for 
		any $x^0\in\R^4$.  The left part of Figure \ref{refhinh} shows a plot of the solution for a specific initial condition.
	\end{example}


	\begin{figure}[htb]
      			\includegraphics[width=0.45\textwidth]{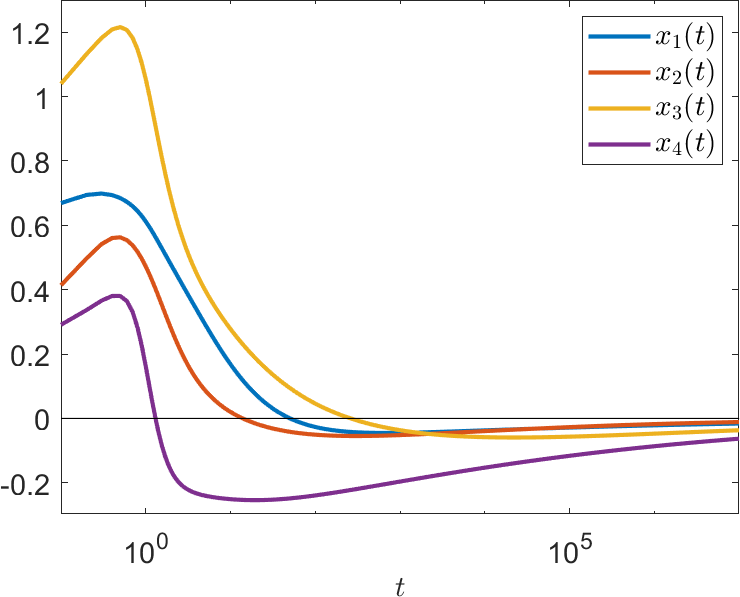}
			\hfill
			\includegraphics[width=0.45\textwidth]{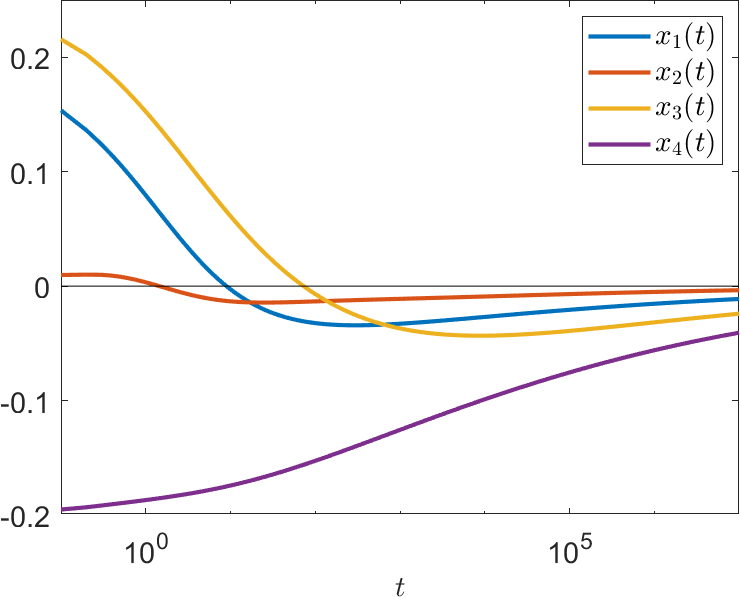} 
		\caption{\emph{Left:} Trajectories of the solution of \eqref{Eq54} with the initial condition 
			$x^0=(0.5, -0.3, 0.7, -0.4)^{\mathrm T}$. 
			\emph{Right:} Trajectories of the solution of \eqref{Eq58} with the initial condition 
			$x^0=(0.2,-0.1,0.3,-0.25)^{\mathrm T}$.
			As in Figure \ref{refhinh1}, the horizontal axes in both plots are in a logarithmic scale.
			{Both numerical solutions have been computed with Garrappa's implementation
			of the trapezoidal algorithm 
			mentioned in Remark \ref{rmk:numsol} using the step size $h = 0.1$}.}
		\label{refhinh}
	\end{figure}
     
	\begin{example}
        	Let us consider the system 
		\begin{align}
			^C D^{\hat\alpha}_{0^+} x(t)& = Ax(t)+f(x(t)),\; t > 0, \label{Eq58}\\
			x(0)&=x^0\in \R^4, \label{incond2}
		\end{align}
		where $A$ and $\hat \alpha$ are as in Example \ref{ex:61}
         	and $f(x(t)) = (f_1(x(t)), f_2(x(t)), f_3(x(t)), f_4(x(t))^{\mathrm T}$ with 
         	$f_1(x) = x_1^{2}+x_2^3-x_4^3$, $f_2(x) = 3x_1^{2}+4x_2^3-5x_4^4$, 
         	$f_3(x)=f_4(x) = x_1^{3}+3x_2^3$ for $x=(x_1,\dots,x_4)^{\rm T}\in \R^4$. 
         	Due to the fact that $\sigma_{\hat\alpha}(A) \subset \mathbb C_{-}$,
         	Theorem \ref{main_result_2} asserts that the system \eqref{Eq58} is Mittag-Leffer stable. 
         	Furthermore, when the initial value vector $x^0$ is close enough to the origin, 
         	its solution $\varphi(\cdot,x^0)$ converges to the origin at a rate no slower than 
         	$t^{-\frac{45}{47\sqrt{33}}}$ as $t \to \infty$.  We provide plots of a solution in 
         	the right part of Figure \ref{refhinh}.
	\end{example}
        

        To further illustrate the range of applicability of our results, we conclude with two more examples
        that have also been investigated with completely different methods elsewhere~\cite{DH}. The 
        fundamental difference between these following examples on the one hand and the examples 
        discussed so far on the other hand is that we now look at coefficient matrices~$A$ where some of the diagonal entries
        are zero (Example \ref{ex:dh1}) or even positive (Example \ref{ex:dh2}) while in the earlier examples all diagonal entries 
        had been negative.
        
        \begin{example}
        	\label{ex:dh1}
		We consider the linear homogeneous system~\eqref{eq1}
		with $\hat \alpha = (2/5, 3/10, 1/2)^{\mathrm T}$
		and 
		\[
			A = \begin{pmatrix}
				-3 & 0 & 1.5 \\
				-0.5 & 0 & 0.5 \\
				6 & -1 & -3    
				\end{pmatrix} .
		\]
		For this problem, we may apply Theorem \ref{dl1} and find that $m = 10$, 
		i.e.\ $\gamma = 1/10$, and $\hat p = (4, 3, 5)^{\mathrm T}$.		
		Thus, the matrix $B$ is of size $(12 \times 12)$. 
		Taking into consideration that, in the notation of Section \ref{s3},
		$\det A_{(2)} = \det A_{(3)} = \det A_{(1,3)} = 0$, the nonzero elements of its 
		rightmost column are 
		$(B)_{1, 12} = -b_0 = \det A = -3/4$, 
		$(B)_{5, 12} = -b_4 = - \det A_{(1)} = -1/2$,
		$(B)_{8, 12} = -b_7 = \det A_{(1,2)} = -3$ and
		$(B)_{9, 12} = -b_8 = \det A_{(2,3)} = -3$.
		The eigenvalues $\lambda_k$ of $B$ are plotted in the left part of Figure \ref{fig:dh1} from which 
		one can see that the property $|\arg \lambda_k | > \pi \gamma / 2$ for all $k$, so that
		the system is asymptotically stable. 
		A plot of one particular solution is shown in the right part of Figure \ref{fig:dh1}.
		Here, the asymptotics can be seen to set in much earlier than in the previous examples.
	\end{example}
        
	\begin{figure}[htb]
			\includegraphics[height=0.42\textwidth]{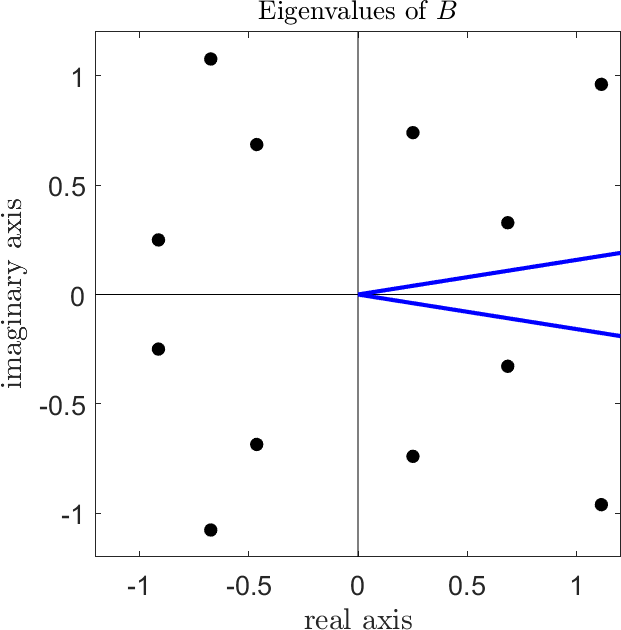}
			\hfill
			\includegraphics[height=0.42\textwidth]{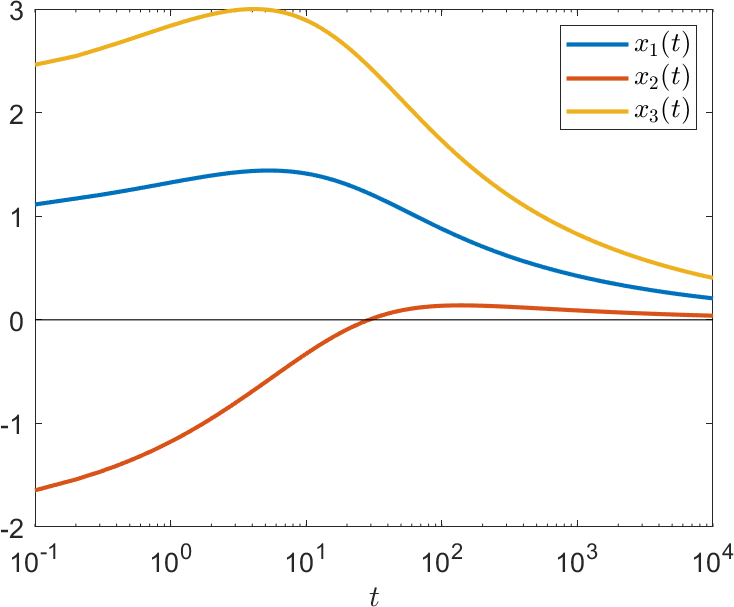}
		\caption{\emph{Left:} Location of the eigenvalues of the matrix $B$ from Example \ref{ex:dh1} in the complex plane.
			The blue rays are oriented at an angle of $\pm \gamma \pi / 2 = \pm \pi/20$ from the 
			positive real axis and hence indicate the boundary of the critical sector 
			$\{ z \in \mathbb C : |\arg z | \le \gamma \pi /2\}$.
			Since all eigenvalues are outside of this sector, we can derive the asymptotic stability of the system.
			\emph{Right:} Trajectories of the solution of Example \ref{ex:dh1} with the initial condition 
			$x^0 = (1, -2, 2)^{\mathrm T}$, {numerically computed with the same algorithm as in the other
			examples with a step size of $h = 0.1$}.}
		\label{fig:dh1}
	\end{figure}


        \begin{example}
        	\label{ex:dh2}
        	In our last example, we consider the linear homogeneous system~\eqref{eq1}
		with
		\[
			A = \begin{pmatrix}
				-1 & 1 & 0 \\
				0.25 & -2 & 1 \\
				-2 & 0 & 1    
				\end{pmatrix} 
		\]
		and  $\hat \alpha = (1/2, 2/5, 3/10)^{\mathrm T}$.
		For this problem, we may also apply Theorem \ref{dl1} and find that $m = 10$, 
		i.e.\ $\gamma = 1/10$, and $\hat p = (5, 4, 3)^{\mathrm T}$.		
		Thus, the matrix $B$ is again of size $(12 \times 12)$, and
		the nonzero elements of its 	rightmost column are, once more using the 
		notation of Section \ref{s3},
		$(B)_{1, 12} = -b_0 = \det A = -1/4$, 
		$(B)_{4, 12} = -b_3 = - \det A_{(3)} = -7/4$,
		$(B)_{5, 12} = -b_4 = -\det A_{(2)} = 1$, 
		$(B)_{6, 12} = -b_5 = -\det A_{(1)} = 2$, 
		$(B)_{8, 12} = -b_7 = \det A_{(2,3)} = -1$,
		$(B)_{9, 12} = -b_8 = \det A_{(1,3)} = -2$ and
		$(B)_{10, 12} = -b_9 = \det A_{(1,2)} = 1$.
		The eigenvalues $\lambda_k$ of $B$ are plotted in the left part of Figure \ref{fig:dh2} from which 
		one can see that the property $|\arg \lambda_k | > \pi \gamma / 2$ for all $k$, so that
		the system is asymptotically stable. 
		A plot of one particular solution is shown in the right part of  Figure~\ref{fig:dh2}.
        \end{example} 

	\begin{figure}[htb]
			\includegraphics[height=0.4\textwidth]{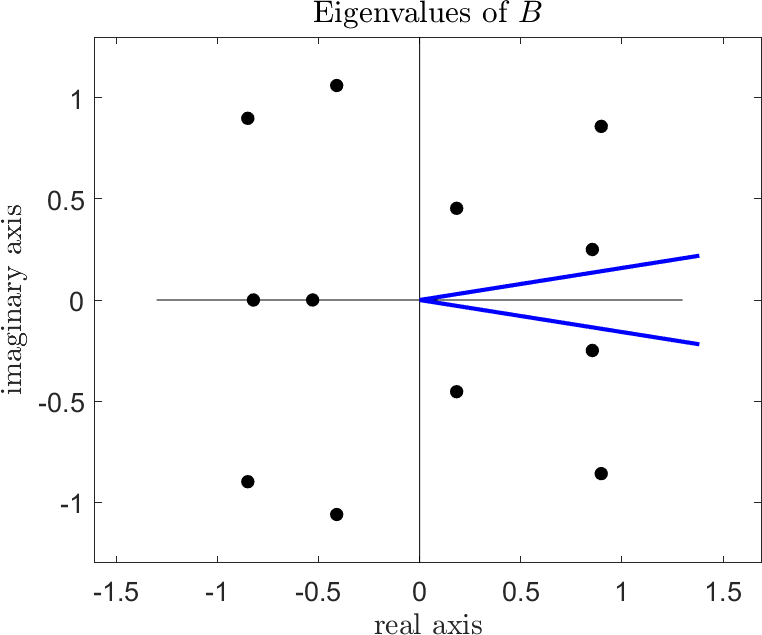}
			\hfill
			\includegraphics[height=0.4\textwidth]{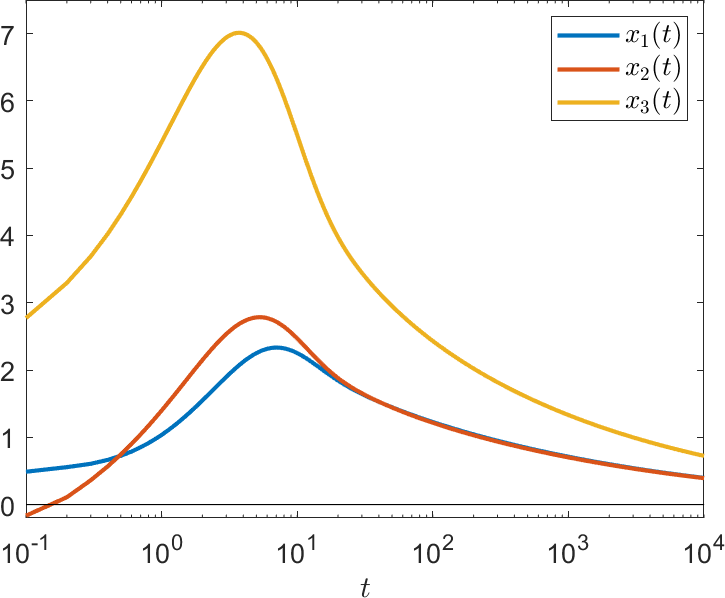}
		\caption{\emph{Left:} Location of the eigenvalues of the matrix $B$ from Example \ref{ex:dh2} in the complex plane.
			The blue rays are oriented at an angle of $\pm \gamma \pi / 2 = \pm \pi/20$ from the 
			positive real axis and hence indicate the boundary of the critical sector 
			$\{ z \in \mathbb C : |\arg z | \le \gamma \pi /2\}$.
			Since all eigenvalues are outside of this sector, the system is asymptotically stable.
			\emph{Right:} Trajectories of the solution of Example \ref{ex:dh2} with the initial condition 
			$x^0 = (1, -2, 2)^{\mathrm T}${, again computed with the same numerical method and a step size $h=0.1$}.}
		\label{fig:dh2}
	\end{figure}

%

\section{Conclusions}

{%
For a very large class of incommensurate fractional differential equation systems, 
we have developed an algorithm that can effectively determine whether or not the
given system is stable. In contrast to earlier methods, our algorithm only requires input data information that is
readily available in practice. The method is general in the sense that it works independently of whether the orders 
of the individual differential equations are rational or irrational.}

{%
If all orders are rational then our approach comprises the application of Theorem \ref{dl1} to the coefficient matrix of the
system, thus constructing an auxiliary matrix $B$, and finding out the locations of the eigenvalues (in the classical
sense) of this matrix $B$, which is a standard task that can be solved by classical techniques from numerical
linear algebra. Having done this, Theorem \ref{dl1} then immediately allows to draw the desired conclusions about the 
system's stabilty properties from the eigenvalues of $B$.}

{%
In the case when some or all equations of the system have irrational orders, it is necessary to apply our Algorithm 2
(which, in turn, uses Algorithm 1) first to obtain a rational approximation of the given system to which we then apply 
the scheme outlined above. As indicated in Subsection \ref{subs:rat-equiv}, the initial step of this process requires 
to find suitable lower bounds for the quantity $\delta^2_{\hat \alpha}(A)$. For this task, Proposition \ref{md1}
provides a general solution under certain assumptions on the system's coefficient matrix. If the matrix does not have
the required properties then an individual investigation is currently necessary. The search for suitable bounds for
$\delta^2_{\hat \alpha}(A)$ under less restrictive assumption is a relevant question for future research.}
        	
\section*{Acknowledgments}
%
The research of H. D. Thai and H. T. Tuan is supported by the Vietnam National Foundation for Science and Technology Development (NAFOSTED) under grant number {\bf 101.02--2021.08}.


\begin{thebibliography}{11}
%
%
%
%
	\bibitem{Cong23} 
	N. D. Cong, 
	Semigroup property of fractional differential operators and its applications. 
	{\em Discrete and Continuous Dynamical Systems - Series B},
	{\bf 28} (2023), pp. 1--19.

	\bibitem{Deng}
	W. Deng, C. Li and J. L\"u, 
	Stability analysis of linear fractional differential system with multiple time delays. 
	{\em Nonlinear Dynamics}, {\bf 48} (2007), pp. 409--416.
%
	\bibitem{Kai}
	K.~Diethelm,	 
	{\em The Analysis of Fractional Differential Equations.}
	Berlin: Springer, 2010.
	
	\bibitem{DH}
	K. Diethelm and S. Hashemishahraki,
	Stability properties of multi-order fractional differential systems in 3D.
	To appear in \emph{Proceedings of the International Conference on Fractional Differentiation and its Applications 2024.}
	Preprint: arXiv:2312.10653.
	
	\bibitem{Tuan2017}
	K. Diethelm, S. Siegmund and H. T. Tuan,
	Asymptotic behavior of solutions of linear multi-order fractional differential equation systems. 
	{\em Fractional Calculus and Applied Analysis}, {\bf 20} (2017), pp. 1165--1195.
%
	\bibitem{TuanThai_2022}
	K. Diethelm, H. D. Thai and H. T. Tuan, 
	Asymptotic behaviour of solutions to non-commensurate 
	fractional-order planar systems. 
	{\em Fractional Calculus and Applied Analysis}, {\bf 25} (2022), pp. 1324--1360.
	
	\bibitem{Ga2015}
	R. Garrappa, 
	Trapezoidal methods for fractional differential equations: Theoretical and computational aspects.
	{\em Mathematics and Computers in Simulation}, {\bf 110} (2015), pp.~96--112.
	 
	\bibitem{Ga2018}
	R. Garrappa, 
	Numerical solution of fractional differential equations: A survey and a software tutorial. 
	{\em Mathematics}, {\bf 6} (2018), Article No.~16.
	
%
%
	\bibitem{HINRICHSEN}
	D. Hinrichsen and A. J. Pritchard, 
	Stability radii of linear systems. 
	{\em Systems and Control Letters}, {\bf 7} (1986), Issue 1, pp. 1--10.
	%

	%
	\bibitem{Jia}
{
J. Jia, F. Wang, and Z. Zeng, Global stabilization of fractional-order memristor-based neural networks with incommensurate orders and multiple time-varying delays: a positive-system-based approach. {\em Nonlinear Dynamics}, {\bf 104} (2021), pp. 2303--2329.
}	%
	\bibitem{Aguliar1}
 {
 D. Melchor-Aguilar and J. Mendiola-Fuentes, {Modification of Mikhailov stability criterion for fractional commensurate order systems.} {\em Journal of The Franklin Institute}, {\bf 355} (2018), pp. 2779--2790.}	

	
%
%
     \bibitem{BS_22}
{
B. K. Lenka and S. N. Bora, New global asymptotic stability conditions for a class of nonlinear time-varying fractional systems. {\em European Journal of Control}, {\bf 63} (2022), pp. 97--106.
%
}
	\bibitem{Petras09}
	I. Petráš, 
	Stability of fractional-order systems with rational orders: A survey. 
	{\em Fractional Calculus and Applied Analysis}, {\bf 12} (2009), pp. 269--298.
	
%
%
	\bibitem{Radwan}
	A. G. Radwan, A. M. Soliman, A. S. Elwakil and A. Sedeek,
	On the stability of linear systems with fractional-order elements. 
	{\em Chaos, Solitons \& Fractals}, {\bf 40} (2009), no. 5, pp. 2317--2328.

	\bibitem{Kostic}
	E. \v Sanca, V. R. Kosti\'c and L. Cvetkovi\'c, 
	Fractional pseudo-spectra and their localizations. 
	{\em Linear Algebra and its Applications}, 
	{\bf 559} (2018), pp. 244--269.
	%
	\bibitem{Sabatier}
{
J. Sabatier, C. Farges and J.-C. Trigeassou, A stability test for non-commensurate fractional order systems.
{\em Systems and Control Letters,} {\bf 62} (2013), no. 9, pp. 739--746.
}
	%
	\bibitem{Stanis}
	R. Stanisławski, 
	Modified Mikhailov stability criterion for continuous-time noncommensurate fractional-order systems. 
	{\em Journal of the Franklin Institute}, 
	{\bf 359} (2022), pp. 1677--1688.
	
%
%
	\bibitem{Trefethen} 
	L. N. Trefethen and M. Embree, 
	{\em Spectra and Pseudospectra: The Behavior of Nonnormal Matrices and Operators.} 
	Princeton: Princeton University Press, 2005.
	%
	\bibitem{Trige}
{
J. Trigeassou, A. Benchellal, N. Maamri and  T. Poinot, A frequency approach to the stability of fractional differential equations. {\em Transactions on Systems, Signals and Devices}, {\bf 4} (2009), no. 1, pp. 1--25.
}
	%
	\bibitem{Tuan}
	H. T. Tuan and H. Trinh, 
	Global attractivity and asymptotic stability of mixed-order fractional systems. 
	{\em IET Control Theory and Applications}, 
	{\bf 14} (2020), pp. 1240--1245.
%
	\bibitem{Tuan-Thinh}
{
H. T. Tuan and L. V. Thinh, Qualitative analysis of solutions to mixed-order positive linear coupled systems with  bounded or unbounded delays. {\em ESAIM - Control, Optimisation and Calculus of Variations}, {\bf 29} (2023), pp. 1--35.
}%
\bibitem{Tuan-Thinh2}
{
L. V. Thinh and H. T. Tuan, Separation of solutions and the attractivity of fractional-order positive linear delay systems with variable coefficients. {\em Communications in Nonlinear Science and Numerical Simulation}, {\bf 132} (2024), Article No.\ 10789.
}
%
	\bibitem{VanLoan}
	C. F. Van Loan, 
	How near is a stable matrix to an unstable matrix?
	{\em Linear Algebra and its Role in Systems Theory}, R. A. Brualdi, D. H. Carlson, B. N. Datta, C.~R. Johnson
	and R. J. Plemmons (eds.), Providence: Amer. Math. Soc., 1985, pp. 465--478.
	
	
\end{thebibliography}
\end{document}